\documentclass{amsart}
\usepackage{amsmath,amsthm,amsfonts,amssymb,mathrsfs}
\date{\today}

\usepackage{color}

\usepackage{hyperref}

\newtheorem{theorem}{Theorem}[section]

\newtheorem{proposition}[theorem]{Proposition}
\newtheorem{corollary}[theorem]{Corollary}
\newtheorem{lemma}[theorem]{Lemma}
\theoremstyle{definition}
\newtheorem{example}[theorem]{Example}
\newtheorem{remark}[theorem]{Remark}

\newtheorem{definition}[theorem]{Definition}

\begin{document}

\title[On the Brandt $\lambda^0$-extensions of
monoids with zero] {On the Brandt
$\lambda^0$-extensions of monoids with zero}

\author{Oleg ~Gutik}
\address{Department of Mathematics, Ivan Franko Lviv National University,
Universytetska 1, Lviv, 79000, Ukraine}
\email{o\underline{\hskip5pt}\,gutik@franko.lviv.ua,
ovgutik@yahoo.com}

\author{Du\v{s}an ~Repov\v{s}}
\address{Faculty of Mathematics and Physics,
and Faculty of Education, University of Ljubljana,
P.~O.~B. Jadranska 19, Ljubljana, 1000, Slovenia}
\email{dusan.repovs@guest.arnes.si}



\begin{abstract}
We study algebraic properties of the Brandt $\lambda^0$-extensions of
monoids with zero and non-trivial homomorphisms between the Brandt
$\lambda^0$-extensions of monoids with zero. 
We introduce
finite, compact topological Brandt $\lambda^0$-extensions of
topological semigroups and countably compact topological Brandt
$\lambda^0$-extensions of topological inverse semigroups in the
class of topological inverse semigroups
and establish the structure
of such extensions and non-trivial continuous homomorphisms
between such topological Brandt $\lambda^0$-extensions of
topological  monoids with zero.
We
also
describe a category whose
objects are ingredients in the constructions of finite (compact,
countably compact) topological Brandt $\lambda^0$-extensions of
topological  monoids with zeros.
\end{abstract}

\maketitle

\section{Introduction and preliminaries}

We shall follow the terminology of \cite{BucurDeleanu1968,
CliffordPreston1961-1967, Petrich1984}. 
Given a semigroup $S,$ we shall
denote the set of idempotents of $S$ by $E(S)$.
A
semigroup $S$ with the adjoined unit [zero] will be denoted by
$S^1$ [$S^0$] (cf. 
\cite{CliffordPreston1961-1967}). 
Next, we shall denote the unit
(identity) and the zero of a semigroup $S$ by $1_S$ and $0_S$,
respectively. 
Given a subset $A$ of a semigroup $S$, we shall denote
by
$A^*=A\setminus\{ 0_S\}$ and $|A|=$ the cardinality of $A$. 
A
semigroup $S$ is 
called
\emph{regular} if for any $x\in S$ there
exists $y\in S$ such that $xyx=x$,
and it is called \emph{inverse}
if for any $x\in S$ there exists a unique $y\in S$ such that
$xyx=x$ and $yxy=y$.
Such an element $y$ is called \emph{inverse}
of $x$ and it
is denoted by $x^{-1}$. 
An inverse semigroup $S$ is
called \emph{Clifford} if $xx^{-1}=x^{-1}x$, for all $x\in S$. 
We
note that $xx^{-1}$ is an idempotent in $S$ for any $x\in S$, and
that for any Clifford inverse semigroup $S$, every idempotent is in
the center of $S$.

If $h\colon S\rightarrow T$ is a homomorphism (or a map) from a
semigroup $S$ into a semigroup $T$ and if $s\in S$,
then we denote
the image of $s$ under $h$ by $(s)h$. A semigroup homomorphism
$h\colon S\rightarrow T$ is called \emph{trivial} if $(s)h=(t)h$
for all $s, t\in S$. A~semigroup $S$ is called {\it
congruence-free} if it has only two congruences: 
the identical and
the universal~\cite{Petrich1984}. Obviously, a semigroup $S$ is
congruence-free if and only if every homomorphism $h$ of $S$ into
an arbitrary semigroup $T$ is an isomorphism ``into'' or is a
trivial homomorphism.

Let $S$ be a semigroup with zero and $I_\lambda$ 
a set of
cardinality $\lambda\geqslant 1$. We define the semigroup
operation on the set $B_{\lambda}(S)=(I_\lambda\times S\times
I_\lambda)\cup\{ 0\}$ as follows:
\begin{equation*}
 (\alpha,a,\beta)\cdot(\gamma, b, \delta)=
  \begin{cases}
    (\alpha, ab, \delta), & \text{ if } \beta=\gamma; \\
    0, & \text{ if } \beta\ne \gamma,
  \end{cases}
\end{equation*}
and $(\alpha, a, \beta)\cdot 0=0\cdot(\alpha, a, \beta)=0\cdot
0=0,$ for all $\alpha, \beta, \gamma, \delta\in I_\lambda$ and $a,
b\in S$. If $S=S^1$ then the semigroup $B_\lambda(S)$ is called
the {\it Brandt $\lambda$-extension of the semigroup}
$S$~\cite{Gutik1999, GutikPavlyk2001}. Obviously, if $S$ has zero
then ${\mathcal J}=\{ 0\}\cup\{(\alpha, 0_S, \beta)\mid 0_S$ is
the zero of $S\}$ is an ideal of $B_\lambda(S)$. We put
$B^0_\lambda(S)=B_\lambda(S)/{\mathcal J}$ and the semigroup
$B^0_\lambda(S)$ is called the {\it Brandt $\lambda^0$-extension
of the semigroup $S$ with zero}~\cite{GutikPavlyk2006}. 

Next, if
$A\subseteq S$ then we shall denote $A_{\alpha\beta}=\{(\alpha, s,
\beta)\mid s\in A \}$ if $A$ does not contain zero, and
$A_{\alpha,\beta}=\{(\alpha, s, \beta)\mid s\in A\setminus\{ 0\}
\}\cup \{ 0\}$ if $0\in A$, for $\alpha, \beta\in I_{\lambda}$. If
$\mathcal{I}$ is a trivial semigroup (i.e. $\mathcal{I}$ contains
only one element), then we denote the semigroup $\mathcal{I}$ with
the adjoined zero by ${\mathcal{I}}^0$. Obviously, for any
$\lambda\geqslant 2$, the Brandt $\lambda^0$-extension of the
semigroup ${\mathcal{I}}^0$ is isomorphic to the semigroup of
$I_\lambda\times I_\lambda$-matrix units and any Brandt
$\lambda^0$-extension of a semigroup with zero which also contains
a non-zero idempotent contains the semigroup of $I_\lambda\times
I_\lambda$-matrix units. 

We shall denote
  the
semigroup of $I_\lambda\times I_\lambda$-matrix units by $B_\lambda$ 
and 
the subsemigroup of $I_\lambda\times
I_\lambda$-matrix units of the Brandt $\lambda^0$-extension of a
monoid $S$ with zero by
$B^0_\lambda(1)$. We always consider the Brandt
$\lambda^0$-extension only of a monoid with zero. Obviously, for
any monoid $S$ with zero we have $B^0_1(S)=S$. Note that every
Brandt $\lambda$-extension of a group $G$ is isomorphic to the
Brandt $\lambda^0$-extension of the group $G^0$ with adjoined
zero. The Brandt $\lambda^0$-extension of the group with adjoined
zero is called a
\emph{Brandt
semigroup}~\cite{CliffordPreston1961-1967, Petrich1984}. A
semigroup $S$ is a Brandt semigroup if and only if $S$ is a
completely $0$-simple inverse semigroup~\cite{Clifford1942,
Munn1957} (cf.  also \cite[Theorem~II.3.5]{Petrich1984}). 
We also
observe that the semigroup $B_\lambda$ of $I_\lambda\times
I_\lambda$-matrix units is isomorphic to the Brandt
$\lambda^0$-extension of the two-element monoid with zero $S=\{
1_S, 0_S\}$ and the
trivial semigroup $S$ (i.~e. $S$ is a singleton
set) is isomorphic to the Brandt $\lambda^0$-extension of $S$ for
every cardinal $\lambda\geqslant 1$. We shall say that the Brandt
$\lambda^0$-extension $B_\lambda^0(S)$ of a semigroup $S$ is
\emph{finite} if the cardinal $\lambda$ is finite.

In this paper we establish homomorphisms of the Brandt
$\lambda^0$-extensions of monoids with zeros.
We
also
describe a
category whose objects are ingredients in the constructions of
the Brandt $\lambda^0$-extensions of monoids with zeros. We introduce
finite, compact topological Brandt $\lambda^0$-extensions of
topological semigroups and countably compact topological Brandt
$\lambda^0$-extensions of topological inverse semigroups in the
class of topological inverse semigroups,
and establish the structure
of such extensions and non-trivial continuous homomorphisms
between such topological Brandt $\lambda^0$-extensions of
topological  monoids with zero.
We
also
describe a category whose
objects are ingredients in the constructions of finite (compact,
countably compact) topological Brandt $\lambda^0$-extensions of
topological  monoids with zeros.


\section{Some properties of the Brandt $\lambda^0$-extensions of
semigroups}

Gutik and Pavlyk \cite{GutikPavlyk2006}  proved
that for every
cardinal $\lambda\geqslant 1,$ the Brandt $\lambda^0$-extension of
a semigroup $S$ is a regular, orthodox, inverse, $0$-simple or
completely $0$-simple semigroup if and only if such
is also $S$. They also
proved that for every cardinal $\lambda\geqslant 1$, the
Brandt $\lambda^0$-extension of a semigroup $S$ with zero is a
congruence-free semigroup if and only if such is also $S$.
The definition of the semigroup operation on the Brandt
$\lambda^0$-extension of a semigroup implies the following:

\begin{proposition}\label{proposition1-1}
Let $\lambda\geqslant 1$ be any cardinal. Then:
\begin{itemize}
    \item[$(i)$] If $T$ is a subsemigroup of a semigroup $S$ then
            $B_{\lambda}^0(T)$ is a subsemigroup of
            $B_{\lambda}^0(S)$; and
    \item[$(ii)$] If $T$ is a left (resp., right, two-sided) ideal
            of a semigroup $S$ then $B_{\lambda}^0(T)$ is a left
            (resp., right, two-sided) ideal in $B_{\lambda}^0(S)$.
\end{itemize}
\end{proposition}

\begin{proposition}\label{proposition1-3}
Let $\lambda_1,\lambda_2\geqslant 1$ be any cardinals, $S$ a
semigroup and $S_1$ the Brandt $\lambda^0_1$-extension of the
semigroup $S$. Then the Brandt $\lambda^0_2$-extension of
the
semigroup $S_1$ is isomorphic to
the
Brandt $\lambda^0$-extension of
the
semigroup $S$ for the cardinal $\lambda=\lambda_1\cdot\lambda_2$.
\end{proposition}

\begin{proof}
Let $I_{\lambda_1}$ and $I_{\lambda_2}$ be the sets of cardinality
$\lambda_1$ and $\lambda_2$, respectively. We put
$I_\lambda=I_{\lambda_1}\times I_{\lambda_2}$. Then the set
$I_\lambda$ has the cardinality $\lambda=\lambda_1\cdot\lambda_2$.
We define a map $h\colon
B_{\lambda_2}^0(S_1)=B_{\lambda_2}^0(B_{\lambda_1}^0(S))\rightarrow
B_{\lambda}^0(S)$ as follows
\begin{equation*}
    (\alpha_2,(\alpha_1,s,\beta_1),\beta_2)h=
    ((\alpha_2,\alpha_1),s,(\beta_1,\beta_2)) \qquad\qquad
    \mbox{and} \qquad\qquad (0_2)h=0,
\end{equation*}
for $s\in S$, $\alpha_1,\beta_1\in I_{\lambda_1}$,
$\alpha_2,\beta_2\in I_{\lambda_2}$, and zeros $0_2$ and $0$ of
semigroups $B_{\lambda_2}^0(B_{\lambda_1}^0(S))$ and
$B_{\lambda}^0(S)$, respectively. Obviously, the map $h\colon
B_{\lambda_2}^0(B_{\lambda_1}^0(S))\rightarrow B_{\lambda}^0(S)$
is bijective and it preserves the semigroup operation, hence $h$
is an isomorphism.
\end{proof}

The cardinal arithmetics and Proposition~\ref{proposition1-3}
imply the following  corollaries:

\begin{corollary}\label{corollary1-4}
Let $S$ be a semigroup and $\lambda_1$, $\lambda_2$  any
infinite cardinals. Then
$B_{\lambda_2}^0(B_{\lambda_1}^0(S))=B_{\lambda}^0(S)$, where
$\lambda=\sup\{\lambda_1,\lambda_2\}$.
\end{corollary}

\begin{corollary}\label{corollary1-5}
$B_{\lambda}^0(B_{\lambda}^0(S))=B_{\lambda}^0(S)$ for every
infinite cardinal $\lambda$ and any semigroup $S$.
\end{corollary}

\begin{corollary}\label{corollary1-6}
Let $\lambda,\nu\geqslant 1$ be any cardinals. Then the Brandt
$\lambda^0$-extension of the semigroup of $I_\nu\times
I_\nu$-matrix units is the semigroup of $I_{\lambda\cdot\nu}\times
I_{\lambda\cdot\nu}$-matrix units.
\end{corollary}

\begin{corollary}\label{corollary1-7}
Let $\lambda$ be any infinite cardinal. Then the Brandt
$\lambda^0$-extension of the semigroup of $I_\lambda\times
I_\lambda$-matrix units is the semigroup of $I_\lambda\times
I_\lambda$-matrix units.
\end{corollary}

\begin{corollary}\label{corollary1-8}
Let $\lambda\geqslant 1$ be any cardinal. Then the Brandt
$\lambda^0$-extension of a Brandt semigroup is a Brandt semigroup.
\end{corollary}

Let $S$ be a semigroup with zero $0_S$ and $\{
S_\alpha\}_{\alpha\in\mathscr{A}}$ a family of subsemigroups of
$S$ such that $S=\bigcup_{\alpha\in\mathscr{A}}S_\alpha$ and
$S_\alpha\cap S_\beta=S_\alpha\cdot S_\beta=0_S$ for all distinct
$\alpha,\beta\in\mathscr{A}$. Then the semigroup $S$ is called an
\emph{orthogonal sum} of the semigroups $\{
S_\alpha\}_{\alpha\in\mathscr{A}}$ and it is denoted by
$\sum_{\alpha\in\mathscr{A}}S_\alpha$
(cf. \cite{Petrich1984}).

\begin{proposition}\label{proposition1-9}
Let $\lambda\geqslant 1$ be any cardinal. Let a semigroup $S$ be
an orthogonal sum of a family of semigroups $\{
S_\alpha\}_{\alpha\in\mathscr{A}}$ with zeros. Then the Brandt
$\lambda^0$-extension $B_{\lambda}^0(S)$ of the semigroup $S$ is
isomorphic to the orthogonal sum
$\sum_{\alpha\in\mathscr{A}}B_\lambda^0(S_\alpha)$.
\end{proposition}

\begin{proof}
Obviously,
$B_{\lambda}^0(S)=\bigcup_{\alpha\in\mathscr{A}}B_{\lambda}^0(S_\alpha)$
and $B_{\lambda}^0(S_\alpha)\cap B_{\lambda}^0(S_\beta)=\{ 0\}$
for all distinct $\alpha,\beta\in\mathscr{A}$, where $0$ is the
zero
of the Brandt $\lambda^0$-extension $B_{\lambda}^0(S)$ of the
semigroup $S$. Proposition~\ref{proposition1-1} implies that
$B_{\lambda}^0(S_\alpha)$ is a subsemigroup of $B_{\lambda}^0(S),$
for all $\alpha\in\mathscr{A}$. The semigroup operation in
$B_{\lambda}^0(S)$ implies that for every distinct
$\alpha,\beta\in\mathscr{A}$ and for any non-zero elements
$(\gamma,s_\alpha,\delta)\in B_{\lambda}^0(S_\alpha)$ and
$(\mu,t_\beta,\nu)\in B_{\lambda}^0(S_\beta)$ we have
$(\gamma,s_\alpha,\delta)\cdot(\mu,t_\beta,\nu)=0$.
This completes the proof of the
proposition.
\end{proof}

The semigroup operation on a semigroup $S$ with
$E(S)\neq\varnothing$ induces the \emph{natural partial order}
$\leqslant$ on $E(S)$: $e\leqslant f$ if and only if $ef=fe=e$,
for $e,f\in E(S)$. If $E(S)$ has a
zero then an idempotent
$e\in(E(S))^*$ is called \emph{primitive} if it is minimal in
$(E(S))^*$ (cf. \cite{Petrich1984}).

An inverse semigroup $S$ with zero is called \emph{primitive
inverse}, if every non-zero idempotent of $S$ is
primitive~\cite{Petrich1984}. Since every primitive inverse
semigroup is an orthogonal sum of Brandt semigroups (cf. 
Theorem~II.4.3 of
\cite{Petrich1984}),
Proposition~\ref{proposition1-9} and Corollary~\ref{corollary1-8}
imply:

\begin{corollary}\label{corollary1-10}
Let $\lambda\geqslant 1$ be any cardinal. Then the Brandt
$\lambda^0$-extension of a semigroup $S$ is a primitive inverse
semigroup if and only if such is also $S$.
\end{corollary}


\section{On homomorphisms of the Brandt $\lambda^0$-extensions of
monoids with zero}

The following proposition is obvious:

\begin{proposition}\label{proposition2-1}
Let $S$ and $K$ be semigroups and let
$\lambda\geqslant 2$. Then the
homomorphism $h\colon B_\lambda^0(S)\rightarrow K$ is trivial if
and only if its restriction $h|_{B_\lambda^0(1)}\colon
B_\lambda^0(1)\rightarrow K$ is a trivial homomorphism.
\end{proposition}

\begin{proposition}\label{proposition2-2}
Let $S$ be a monoid with zero, $\lambda\geqslant 1$ any cardinal,
and $B_\lambda^0(S)$ the Brandt $\lambda^0$-extension of $S$. Then
every non-trivial homomorphic image of $B_\lambda^0(S)$ is the
Brandt $\lambda^0$-extension of some monoid with zero. Moreover,
if $T$ is the image of $B_\lambda^0(S)$ under a homomorphism $h,$
then $T$
is
isomorphic to the Brandt $\lambda^0$-extension of the
homomorphic image of the monoid $S_{\alpha,\alpha}$ under the
homomorphism $h$ for any $\alpha\in I_\lambda$.
\end{proposition}

\begin{proof}
In the case $\lambda=1$ the proof is trivial. Therefore we may
assume that $\lambda\geqslant 2$.
Let $T$ be a semigroup and $h\colon B_\lambda^0(S)\rightarrow T$ a
homomorphism. Without loss of generality we can assume that the
homomorphism $h\colon B_\lambda^0(S)\rightarrow T$ is a surjective
map. Note that $(0)h=0_T$ is the
zero in $T$, where $0$ is the
zero in
$B_\lambda^0(S)$. By Theorem~1
of
\cite{Gluskin1955}, the semigroup
$B_\lambda^0(1)$ is congruence-free and thus
Proposition~\ref{proposition2-1} implies that the restriction
$h|_{B_\lambda^0(1)}\colon B_\lambda^0(1)\rightarrow T$ of the
homomorphism $h$ is an isomorphism.

We fix $\alpha_0\in I_\lambda$. Next we shall show that the
semigroup $T$ is the Brandt $\lambda^0$-extension of a semigroup
$T_0$, where $T_0$ is the homomorphic image of the subsemigroup
$S_{\alpha_0,\alpha_0},$ under the homomorphism $h$. For every
$\alpha,\beta\in I_\lambda,$ we denote
$1_{\alpha,\beta}=((\alpha,1_S,\beta))h$ and
$T^*_{\alpha,\beta}=\{((\alpha,s,\beta))h\mid s\in S\}\setminus\{
0_T\}$. First we show that for any $\alpha,\beta,\gamma,\delta\in
I_\lambda$ we have
$\big|T^*_{\alpha,\beta}\big|=\big|T^*_{\gamma,\delta}\big|$. 

We
define the maps $\varphi_{(\alpha,\beta)}^{(\gamma,\delta)}\colon
T^*_{\alpha,\beta}\rightarrow T^*_{\gamma,\delta}$ and
$\varphi^{(\alpha,\beta)}_{(\gamma,\delta)}\colon
T^*_{\gamma,\delta}\rightarrow T^*_{\alpha,\beta}$ by the formulae
$(x)\varphi_{(\alpha,\beta)}^{(\gamma,\delta)}=1_{\gamma,\alpha}\cdot
x\cdot 1_{\beta,\delta}$ and
$(x)\varphi^{(\alpha,\beta)}_{(\gamma,\delta)}=1_{\alpha,\gamma}\cdot
x\cdot 1_{\delta,\beta}$. Then for any
$s_{\alpha,\beta}=((\alpha,s,\beta))h\in T^*_{\alpha,\beta}$,
$s\in S\setminus\{ 0\}$, we get
\begin{equation*}
 \begin{split}
    (s_{\alpha,\beta})
    \big(\varphi_{(\alpha,\beta)}^{(\gamma,\delta)}\circ
    \varphi^{(\alpha,\beta)}_{(\gamma,\delta)}\big) & =
    1_{\alpha,\gamma}\cdot 1_{\gamma,\alpha}\cdot
    s_{\alpha,\beta}\cdot 1_{\beta,\delta}\cdot
    1_{\delta,\beta}=\\
        & =((\alpha,1_s,\gamma))h\cdot ((\gamma,1_s,\alpha))h\cdot
        ((\alpha,s,\beta))h\cdot ((\beta,1_s,\delta))h\cdot
        ((\delta,1_s,\beta))h=\\
        & =\big((\alpha,1_s,\gamma)\cdot (\gamma,1_s,\alpha)\cdot
        (\alpha,s,\beta)\cdot(\beta,1_s,\delta)\cdot
        (\delta,1_s,\beta)\big)h=\\
        & =((\alpha,s,\beta))h=s_{\alpha,\beta},
\end{split}
\end{equation*}
and similarly
\begin{equation*}
    (s_{(\gamma,\delta)})
    \big(\varphi^{(\alpha,\beta)}_{(\gamma,\delta)}\circ
    \varphi_{(\alpha,\beta)}^{(\gamma,\delta)}\big)=
    s_{\gamma,\delta}.
\end{equation*}
Hence the compositions
$\varphi_{(\alpha,\beta)}^{(\gamma,\delta)}\circ
\varphi^{(\alpha,\beta)}_{(\gamma,\delta)}\colon
T^*_{\alpha,\beta}\rightarrow T^*_{\alpha,\beta}$ and
$\varphi^{(\alpha,\beta)}_{(\gamma,\delta)}\circ
\varphi_{(\alpha,\beta)}^{(\gamma,\delta)}\colon
T^*_{\gamma,\delta}\rightarrow T^*_{\gamma,\delta}$ are the
identity maps. Therefore the maps
$\varphi_{(\alpha,\beta)}^{(\gamma,\delta)}$ and
$\varphi^{(\alpha,\beta)}_{(\gamma,\delta)}$ are mutually
invertible and hence we have that
$\big|T^*_{\alpha,\beta}\big|=\big|T^*_{\gamma,\delta}\big|=|T_0\setminus\{
0_T\}|$. This implies that $T=I_\lambda\times (T_0\setminus\{
0_T\})\times I_\lambda\cup\{ 0_T\}$.

Elementary calculations shows that for all $s,t\in S\setminus\{
0_S\}$ we have
\begin{itemize}
    \item[(1)] $s_{\alpha,\beta}\cdot t_{\beta,\gamma}=
    \left\{%
\begin{array}{ll}
    (st)_{\alpha,\gamma}, & \hbox{ if } st\neq 0_S;\\
    0_T, & \hbox{ if } st=0_S;\\
\end{array}%
\right.    $
    \item[(2)] $s_{\alpha,\beta}\cdot t_{\gamma,\delta}=0_T$
    for $\beta\neq\gamma$; and
    \item[(3)] $s_{\alpha,\beta}\cdot 0_T=0_T\cdot
    s_{\alpha,\beta}=0_T$,
    \end{itemize}
$\alpha,\beta,\gamma,\delta\in I_\lambda$, and hence $T$ is the
Brandt $\lambda^0$-extension of the semigroup $T_0$. This proves
the last assertion
of the proposition.
\end{proof}

Since a homomorphic image of a subgroup is a subgroup,
Propositions~\ref{proposition2-1} and \ref{proposition2-2} imply
the following:

\begin{corollary}\label{corollary2-3}
Every non-trivial homomorphic image of a Brandt semigroup is a
Brandt semigroup.
\end{corollary}

\begin{proposition}\label{proposition2-4}
Let $S$ and $T$ be monoids with zeros, and let $\lambda_1$ and
$\lambda_2$ be any cardinals such that
$\lambda_2\geqslant\lambda_1\geqslant 1$. Let $\sigma\colon
B_{\lambda_1}^0(S)\rightarrow B_{\lambda_2}^0(T)$ be a non-trivial
homomorphism.
Suppose that
the monoid $T$ has
the following properties:
\begin{itemize}
    \item[1)] $T$ does not contain the semigroup of
    $I_{\lambda_1}\times I_{\lambda_1}$-matrix units; and
    \item[2)] $T$ does not contain the semigroup of $2\times 2$-matrix
    units $B_2$ such that the zero of $B_2$ is the zero of $T$.
\end{itemize}
Then the following assertions hold:
\begin{itemize}
    \item[$(i)$] The image of zero $0_S$ of the semigroup
    $B_{\lambda_1}^0(S)$ under the homomorphism $\sigma$ is the zero
    of the semigroup $B_{\lambda_2}^0(T)$;
    \item[$(ii)$] If $(\alpha,1_{S},\beta)$ and
    $(\gamma,1_{S},\delta)$ are distinct elements of the Brandt
    $\lambda_1$-extension of the semigroup $S$,
    $\alpha,\beta,\gamma,\delta\in I_{\lambda_1}$, such that
    $((\alpha,1_{S},\beta))\sigma\in
    {T}_{\mu,\nu}$ and $((\gamma,1_{S},\delta))\sigma\in
    {T}_{\iota,\kappa}$ for some $\mu,\nu,\iota,\kappa\in
    I_{\lambda_2}$, then ${T}_{\mu,\nu}^*\cap
    {T}_{\iota,\kappa}^*=\varnothing$.
\end{itemize}
\end{proposition}

\begin{proof}
Suppose to the contrary, that $(0_S)\sigma$ is not the zero of the
semigroup $B_{\lambda_2}^0(T)$. Since the element $(0_S)\sigma$ is
an idempotent of $B_{\lambda_2}^0(T)$, there exists $\alpha\in
I_{\lambda_2}$ such that $(0_S)\sigma\in T^*_{\alpha,\alpha}$.
Since $B_{\lambda_1}^0(1)$ is a congruence-free semigroup,
$\sigma$ is a non-trivial homomorphism and $T$ does not contain
the semigroup of $I_{\lambda_1}\times I_{\lambda_1}$-matrix units.
We can conclude that there exist $\gamma,\delta\in I_{\lambda_2}$
such that $\gamma\neq\alpha$ or $\delta\neq\alpha$ and
$(B_{\lambda_1}^0(S))\sigma\cap
T^*_{\gamma,\delta}\neq\varnothing$. Let
$x\in(B_{\lambda_1}^0(S))\sigma\cap T^*_{\gamma,\delta}$. If
$\gamma\neq\alpha$ then the element $(0_S)\sigma\cdot x$ is the
zero of $B_{\lambda_2}^0(T)$ and if $\delta\neq\alpha$ then the
element $x\cdot (0_S)\sigma$ is the zero of $B_{\lambda_2}^0(T)$.
But $x=(S)\sigma$ for some non-zero element $s$ of the semigroup
$B_{\lambda_1}^0(1)$. Therefore
\begin{equation*}
    x\cdot(0_S)\sigma=(s\cdot 0_S)\sigma=(0_S)\sigma \qquad
\mbox{and} \qquad
    (0_S)\sigma\cdot x=(0_S\cdot s)\sigma=(0_S)\sigma,
\end{equation*}
a contradiction. Hence the statement $(i)$ holds.

Next we shall show that there does not exist $\alpha_0\in
I_{\lambda_2}$ such that $((\alpha,1_S,\beta))\sigma$,
$((\beta,1_S,\alpha))\sigma\in T^*_{\alpha_0,\alpha_0}$ for
distinct $\alpha,\beta\in I_{\lambda_1}$. Suppose to the contrary.
Then since $T_{\alpha_0,\alpha_0}$ is a subsemigroup in
$B_{\lambda_2}^0(T)$ and $\sigma$ is a non-trivial homomorphism
Proposition~\ref{proposition2-1} implies
\begin{equation*}
    ((\alpha,1_S,\beta)))\sigma\cdot((\beta,1_S,\alpha))\sigma=
    ((\alpha,1_S,\beta)\cdot(\beta,1_S,\alpha))\sigma=
    ((\alpha,1_S,\alpha))\sigma\in T^*_{\alpha_0,\alpha_0}
\end{equation*}
and
\begin{equation*}
    ((\beta,1_S,\alpha))\sigma\cdot((\alpha,1_S,\beta)))\sigma=
    ((\beta,1_S,\alpha)\cdot(\alpha,1_S,\beta))\sigma=
    ((\beta,1_S,\beta))\sigma\in T^*_{\alpha_0,\alpha_0},
\end{equation*}
and hence
\begin{equation*}
    (0_S)\sigma=((\alpha,1_S,\alpha)\cdot(\beta,1_S,\beta))\sigma=
    ((\alpha,1_S,\alpha))\sigma\cdot((\beta,1_S,\beta))\sigma\in
    T_{\alpha_0,\alpha_0}.
\end{equation*}
Then by statement $(i)$, the element $(0_S)\sigma$ is the zero of
the semigroup $B_{\lambda_2}^0(T)$. This contradicts the
assumption that the monoid $T$ does not contain the semigroup of
$2\times 2$-matrix units $B_2$ such that the zero of $B_2$ is the
zero of $T$.

In the next step we shall show that there does not exist
$\alpha_0\in I_{\lambda_2}$ such that
$((\alpha,1_S,\beta))\sigma\in T^*_{\alpha_0,\alpha_0}$. Otherwise
we would have
\begin{equation*}
    (\alpha,1_S,\beta)\cdot(\beta.1_S,\alpha)=(\alpha,1_S,\alpha),
\end{equation*}
and since the homomorphism $\sigma$ is non-trivial, we would have
\begin{equation*}
    ((\alpha,1_S,\beta))\sigma\cdot((\beta.1_S,\alpha))\sigma=
    ((\alpha,1_S,\alpha))\sigma\in T^*_{\alpha_0,\alpha_0},
\end{equation*}
and hence $((\alpha,1_S,\alpha))\sigma\in
T^*_{\alpha_0,\alpha_0}$. Therefore $((\alpha,1_S,\beta))\sigma\in
T^*_{\alpha_0,\alpha_0}$ and $((\beta.1_S,\alpha))\sigma\in
T^*_{\alpha_0,\alpha_0}$. This contradicts the previous statement.

We shall show that there does not exist two distinct non-zero
idempotents $(\alpha,1_S,\alpha)$ and $(\beta,1_S,\beta)$ in
$B_{\lambda_1}^0(S)$, $\alpha,\beta\in I_{\lambda_1}$, such that
$((\alpha,1_S,\alpha))\sigma, ((\beta,1_S,\beta))\sigma\in
T^*_{\alpha_0,\alpha_0}$ for some $\alpha_0\in I_{\lambda_2}$.
Suppose to the contrary. Then
\begin{equation*}
    (\alpha,1_S,\alpha)=(\alpha,1_S,\beta)\cdot(\beta,1_S,\alpha)
    \qquad \mbox{and} \qquad
    (\beta,1_S,\beta)=(\beta,1_s,\alpha)\cdot(\alpha,1_S,\beta),
\end{equation*}
and hence
\begin{equation*}
    ((\alpha,1_S,\alpha))\sigma=
    ((\alpha,1_S,\beta))\sigma\cdot((\beta,1_S,\alpha))\sigma
    \quad \mbox{and} \quad ((\beta,1_S,\beta))\sigma=
    ((\beta,1_s,\alpha))\sigma\cdot((\alpha,1_S,\beta))\sigma.
\end{equation*}
Since $\sigma$ is a non-trivial homomorphism,
Proposition~\ref{proposition2-1} implies that
$(\alpha,1_S,\beta)\sigma\in T^*_{\alpha_0,\alpha_0}$ and\break
$(\beta,1_s,\alpha)\sigma\in T^*_{\alpha_0,\alpha_0}$. Hence
$(\alpha,1_S,\alpha)\sigma\in T^*_{\alpha_0,\alpha_0}$ and
$(\beta,1_s,\beta)\sigma\in T^*_{\alpha_0,\alpha_0}$. This is in
contradiction with the previous statement.

In order to complete our proof we need to prove that there do not
exist $\mu,\nu\in I_{\lambda_2}$ such that
$((\alpha,1_S,\beta))\sigma$, $((\gamma,1_S,\delta))\sigma\in
T^*_{\mu,\nu}$ for distinct non-idempotent elements
$(\alpha,1_S,\beta)$ and $(\gamma,1_S,\delta)$ from the semigroup
$B_{\lambda_1}(S)$. Suppose to the contrary. We consider only the
case $\alpha\neq\gamma$. In the case $\beta\neq\delta$ the proof
is similar. Then since $\sigma$ is non-trivial homomorphism,
Proposition~\ref{proposition2-1} implies
\begin{equation*}
    ((\alpha,1_S,\alpha))\sigma=
    ((\alpha,1_S,\beta)\cdot(\beta,1_S,\alpha))\sigma=
    ((\alpha,1_S,\beta))\sigma\cdot((\beta,1_S,\alpha))\sigma\in
    T^*_{\mu,\nu}
\end{equation*}
and
\begin{equation*}
    ((\gamma,1_S,\gamma))\sigma=
    ((\gamma,1_S,\delta)\cdot(\delta,1_S,\gamma))\sigma=
    ((\gamma,1_S,\delta))\sigma\cdot((\delta,1_S,\gamma))\sigma\in
    T^*_{\mu,\nu}.
\end{equation*}
But this contradicts the previous statement. The obtained
contradiction implies the statement of the proposition.
\end{proof}

The following example shows that Proposition~\ref{proposition2-4}
fails in the case when the semigroup $T$ contains the
semigroup of $2\times 2$-matrix units $B_2$ such that zero of $T$
is zero of $B_2$.

\begin{example}\label{example2-5}
Let $B_2$ be the semigroup of $2\times 2$-matrix units. Let $S$ be
the semigroup $B_2$ with the adjoined identity and
$I_4=\{1,2,3,4\}$. We define a map $h\colon B_4\rightarrow B_4(S)$
as follows:
\begin{align*}
 (0)h&=0,&  &{}\\
 ((1,1))h&=(1,(1,1),1),& ((2,2))h&=(1,(2,2),1),\\
 ((3,3))h&=(2,(1,1),2),& ((4,4))h&=(2,(2,2),2),\\
 ((1,2))h&=(1,(1,2),1),& ((2,1))h&=(1,(2,1),1),\\
 ((1,3))h&=(1,(1,1),2),& ((3,1))h&=(2,(1,1),1),\\
 ((1,4))h&=(1,(1,2),2),& ((4,1))h&=(2,(2,1),1),\\
 ((2,3))h&=(1,(2,1),2),& ((3,2))h&=(2,(1,2),1),\\
 ((2,4))h&=(1,(2,2),2),& ((4,2))h&=(2,(2,2),1),\\
 ((3,4))h&=(2,(1,2),2),& ((4,3))h&=(2,(2,1),2),
\end{align*}
where by $0$ we denote the zeros of semigroups $B_4$ and $B_4(S)$.
Elementary calculation shows that the map $h\colon B_4\rightarrow
B_4(S)$ is a semigroup homomorphism.
\end{example}

The following example shows that Proposition~\ref{proposition2-4}
fails in the case when the semigroup $T$ contains the
semigroup of $I_{\lambda_1}\times I_{\lambda_1}$-matrix units
$B_{\lambda_1}$.

\begin{example}\label{example2-6}
Let $\lambda_1$ and $\lambda_2$ be any cardinals $\geqslant 2$.
Let $P$ be the semigroup of $I_{\lambda_1}\times
I_{\lambda_1}$-matrix units $B_{\lambda_1}$ with the adjoined
identity $\iota_1$, $0_1$ be the zero of $B_{\lambda_1}$ and let
$z\notin P$. We extend
the
semigroup operation onto $T=P\cup\{ z\}$
as follows:
\begin{equation*}
    s\cdot z=z\cdot s=z\cdot z=z, \qquad \mbox{for all~} s\in P.
\end{equation*}
Obviously, $z$ is the zero of the semigroup $T$.

Let $S$ be a monoid with the zero $0_S$ of cardinality $\geqslant
3$. We define a map $\sigma\colon B_{\lambda_1}^0(S)\rightarrow
B_{\lambda_2}^0(T)$ as follows: fix arbitrary $\alpha\in
I_{\lambda_2}$ and put
\begin{equation*}
    (x)\sigma=
    \left\{%
\begin{array}{ll}
    (\alpha,(\beta,\iota_1,\gamma),\alpha), &
    \hbox{~if~} x=(\beta,s,\gamma)
    \hbox{~is a non-zero element of~}B_{\lambda_1}^0(S);\\
    (\alpha,0_1,\alpha), & \hbox{~if~} x \hbox{~is zero of ~}B_{\lambda_1}^0(S).\\
\end{array}%
\right.
\end{equation*}
Obviously, such a map $\sigma\colon B_{\lambda_1}^0(S)\rightarrow
B_{\lambda_2}^0(T)$ is a semigroup homomorphisms.
\end{example}

\begin{definition}\label{definition2-7}
Let $\lambda$ be any cardinal $\geqslant 2$. We shall say that a
semigroup $S$ has the \emph{$\mathcal{B}^*$-property} if $S$ does
not contain the semigroup of $2\times 2$-matrix units and that $S$
has the \emph{$\mathcal{B}_{\lambda}^*$-property} if $S$ satisfies
the following conditions:
\begin{itemize}
    \item[1)] $T$ does not contain the semigroup of
    $I_\lambda\times I_\lambda$-matrix units; and
    \item[2)] $T$ does not contain the semigroup of $2\times 2$-matrix
    units $B_2$ such that the zero of $B_2$ is the zero of $T$.
\end{itemize}
\end{definition}

Obviously, a semigroup $S$ has the $\mathcal{B}^*$-property if and
only if $S$ has the $\mathcal{B}_2^*$-property, and hence
Proposition~\ref{proposition2-4} implies:

\begin{corollary}\label{corollary2-8}
Let $S$ and $T$ be monoids with zeros, and $\lambda_1$ and
$\lambda_2$ any cardinals such that
$\lambda_2\geqslant\lambda_1\geqslant 1$. Let $\sigma\colon
B_{\lambda_1}^0(S)\rightarrow B_{\lambda_2}^0(T)$ be a non-trivial
homomorphism. If the monoid $T$ has the $\mathcal{B}^*$-property,
then the following assertions hold:
\begin{itemize}
    \item[$(i)$] The image of the zero $0_S$ of the semigroup
    $B_{\lambda_1}^0(S)$ under the homomorphism $\sigma$ is the zero
    of the semigroup $B_{\lambda_2}^0(T)$; and
    \item[$(ii)$] If $(\alpha,1_{S},\beta)$ and
    $(\gamma,1_{S},\delta)$ are distinct elements of the Brandt
    $\lambda_1$-extension of the semigroup $S$,
    $\alpha,\beta,\gamma,\delta\in I_{\lambda_1}$, such that
    $((\alpha,1_{S},\beta))\sigma\in
    {T}_{\mu,\nu}$ and $((\gamma,1_{S},\delta))\sigma\in
    {T}_{\iota,\kappa}$ for some $\mu,\nu,\iota,\kappa\in
    I_{\lambda_2}$, then ${T}_{\mu,\nu}^*\cap
    {T}_{\iota,\kappa}^*=\varnothing$.
\end{itemize}
\end{corollary}

\begin{corollary}\label{corollary2-9}
Let $S$ and $T$ be monoids with zeros, and $\lambda_1$ and
$\lambda_2$ any cardinals such that
$\lambda_2\geqslant\lambda_1\geqslant 1$. Let $\sigma\colon
B_{\lambda_1}^0(S)\rightarrow B_{\lambda_2}^0(T)$ be a non-trivial
homomorphism. Let $\alpha,\beta\in I_{\lambda_1}$ and
$(\alpha,1_S,\beta)\sigma\in T^*_{\gamma,\delta}$.
Suppose that
the
monoid $T$
has
the $\mathcal{B}^*$-property. Then the following
assertions hold:
\begin{itemize}
    \item[$(i)$] If $(\alpha,s,\beta)\sigma$ is a non-zero element
    of the semigroup $B_{\lambda_2}^0(T)$, then
    $(\alpha,s,\beta)\sigma\in T^*_{\gamma,\delta}$; and
    \item[$(ii)$] If $(\alpha,s,\beta)\sigma$ is a non-zero element
    of the semigroup $B_{\lambda_2}^0(T)$, then such is also
    $(\alpha_1,s,\beta_1)\sigma$ for all $\alpha_1,\beta_1\in
    I_{\lambda_1}$.
\end{itemize}
\end{corollary}

\begin{proof}
The statement $(i)$ follows from
Proposition~\ref{proposition1-1}$(ii)$.

Suppose there exist $\alpha_1,\beta_1\in I_{\lambda_1}$ such that
$(\alpha_1,s,\beta_1)\sigma=0_2$ is the zero of the semigroup
$B_{\lambda_2}^0(T)$. Then
\begin{equation*}
\begin{split}
    (\alpha,s,\beta)\sigma=
         & \, ((\alpha,1_S,\alpha_1)\cdot (\alpha_1,s,\beta_1)\cdot
(\beta_1,1_S,\beta))\sigma=\\
       = & \,((\alpha,1_S,\alpha_1))\sigma\cdot
((\alpha_1,s,\beta_1))\sigma\cdot ((\beta_1,1_S,\beta))\sigma=\\
       = & \,((\alpha,1_S,\alpha_1))\sigma\cdot 0_2\cdot
((\beta_1,1_S,\beta))\sigma=0_2.
\end{split}
\end{equation*}
The obtained contradiction implies the assertion
$(ii)$ of the corollary.
\end{proof}

The following theorem describes all non-trivial homomorphisms of
the Brandt $\lambda^0$-extensions of monoids with zeros.

\begin{theorem}\label{theorem2-10}
Let $\lambda_1$ and $\lambda_2$ be cardinals such that
$\lambda_2\geqslant\lambda_1\geqslant 1$. Let $B_{\lambda_1}^0(S)$
and $B_{\lambda_2}^0(T)$ be the Brandt $\lambda_1^0$- and
$\lambda_2^0$-extensions of monoids $S$ and $T$ with zero,
respectively. Let $h\colon S\rightarrow T$ be a homomorphism such
that $(0_S)h=0_T$ and suppose that $\varphi\colon
I_{\lambda_1}\rightarrow I_{\lambda_2}$ is a one-to-one map. Let
$e$ be a non-zero idempotent of $T$, $H_e$ a maximal subgroup of
$T$ with the unity $e$ and $u\colon I_{\lambda_1}\rightarrow H_e$
a map. Then $I_h=\{ s\in S\mid (s)h=0_T\}$ is an ideal of $S$ and
the map $\sigma\colon B_{\lambda_1}^0(S)\rightarrow
B_{\lambda_2}^0(T)$ defined by the formulae
\begin{equation*}
    ((\alpha,s,\beta))\sigma=
    \left\{%
\begin{array}{cl}
    ((\alpha)\varphi,(\alpha)u\cdot(s)h\cdot((\beta)u)^{-1},(\beta)\varphi),
    & \hbox{if}\quad s\notin S\setminus I_h ;\\
    0_2, & \hbox{if}\quad s\in I_h^*,\\
\end{array}%
\right.
\end{equation*}
and $(0_1)\sigma=0_2$ is a non-trivial homomorphism from
$B_{\lambda_1}^0(S)$ into $B_{\lambda_2}^0(T)$. Moreover, if for
the semigroup $T$ the following assertions hold:
\begin{itemize}
    \item[($i$)] Every idempotent of $T$ lies in the center of
    $T$; and
    \item[($ii$)] $T$ has
    the
    $\mathcal{B}_{\lambda_1}^*$-property,
\end{itemize}
then every non-trivial homomorphism from $B_{\lambda_1}^0(S)$ into
$B_{\lambda_2}^0(T)$ can be constructed in this manner.
\end{theorem}

\begin{proof}
A simple verification shows that the set $I_h$ is an ideal in $S$
and that $\sigma$ is a homomorphism from the semigroup
$B_{\lambda_1}^0(S)$ into the semigroup $B_{\lambda_2}^0(T)$.

Let $\sigma\colon B_{\lambda_1}^0(S)\rightarrow
B_{\lambda_2}^0(T)$ be a non-trivial semigroup homomorphism. We
fix $\alpha\in I_{\lambda_1}$. Since
the
homomorphism $\sigma\colon
B_{\lambda_1}^0(S)\rightarrow B_{\lambda_2}^0(T)$ is non-trivial,
$((\alpha,1_S,\alpha))\sigma$ is a non-zero idempotent of
$B_{\lambda_2}^0(T)$, and hence $((\alpha,1_S,\alpha))\sigma=
(\alpha^{\,\prime},e,\alpha^{\,\prime})$ for some $e\in(E(T))^*$
and $\alpha^{\,\prime}\in I_{\lambda_2}$. Let $H_e$ be a maximal
subgroup in $T$ which contains $e$ as a unity and let $G_1$ be the
group of units of $S$. Therefore we have that
$((G_1)_{\alpha,\alpha})\sigma\subseteq
(H_e)_{\alpha^{\,\prime},\alpha^{\,\prime}}$.

Since $(\beta,1_S,\alpha)(\alpha,1_S,\alpha)=(\beta,1_S,\alpha)$
for any $\beta\in I_{\lambda_1}$, we have
\begin{equation*}
  ((\beta,1_S,\alpha))\sigma=
  ((\beta,1_S,\alpha))\sigma\cdot
  (\alpha^{\,\prime},e,\alpha^{\,\prime}),
\end{equation*}
and hence
\begin{equation*}
  ((\beta,1_S,\alpha))\sigma=
  ((\beta)\varphi,(\beta)u,\alpha^{\,\prime}),
\end{equation*}
for some $(\beta)\varphi\in I_{\lambda_2}$ and $(\beta)u\in T$.
Similarly, we have
\begin{equation*}
  ((\alpha,1_S,\beta))\sigma=
  (\alpha^{\,\prime},(\beta)v,(\beta)\psi),
\end{equation*}
for some $(\beta)\psi\in I_{\lambda_2}$ and $(\beta)v\in T$. Since
$(\alpha,1_S,\beta)(\beta,1_S,\alpha)=(\alpha,1_S,\alpha)$, we
have
\begin{equation*}
    (\alpha^{\,\prime},e,\alpha^{\,\prime})=
    ((\alpha,1_S,\alpha))\sigma=
    (\alpha^{\,\prime},(\beta)v,(\beta)\psi)\cdot
    ((\beta)\varphi,(\beta)u,\alpha^{\,\prime})=
    (\alpha^{\,\prime},(\beta)v\cdot(\beta)u,\alpha^{\,\prime}),
\end{equation*}
and hence $(\beta)\varphi=(\beta)\psi=\beta^{\,\prime}\in
I_{\lambda_2}$ and $(\beta)v\cdot(\beta)u=e$. Similarly, since
$(\beta,1_S,\alpha)\cdot(\alpha,1_S,\beta)=(\beta,1_S,\beta)$, we
see that the element
\begin{equation*}
    ((\beta,1_S,\beta))\sigma=
    ((\beta,1_S,\alpha)(\alpha,1_S,\beta))\sigma=
    (\beta^{\,\prime},(\beta)v\cdot(\beta)u,\beta^{\,\prime})
\end{equation*}
is an idempotent, and hence the element $f=(\beta)v\cdot(\beta)u$
is an idempotent of the semigroup $T$. Since idempotents of $T$
lie in the center of $T$, we conclude that
\begin{equation*}
\begin{split}
    (\alpha^{\,\prime},e,\alpha^{\,\prime})=&\,
    ((\alpha,1_S,\alpha))\sigma=
    ((\alpha,1_S,\beta)\cdot(\beta,1_S,\beta)\cdot(\beta,1_S,\alpha))\sigma=\\
    =&\,(\alpha^{\,\prime},(\beta)v,\beta^{\,\prime})\cdot
    (\beta^{\,\prime},f,\beta^{\,\prime})\cdot
    (\beta^{\,\prime},(\beta)u,\alpha^{\,\prime})=\\
    =&\,(\alpha^{\,\prime},(\beta)v\cdot
    f\cdot(\beta)u,\alpha^{\,\prime})= (\alpha^{\,\prime},
    f\cdot(\beta)v\cdot(\beta)u,\alpha^{\,\prime})=\\
    =&\,(\alpha^{\,\prime},f\cdot e,\alpha^{\,\prime})
\end{split}
\end{equation*}
and
\begin{equation*}
\begin{split}
    (\beta^{\,\prime},f,\beta^{\,\prime})=&\,
    ((\beta,1_S,\beta))\sigma=
    ((\beta,1_S,\alpha)\cdot(\alpha,1_S,\alpha)\cdot(\alpha,1_S,\beta))\sigma=\\
    =&\,(\beta^{\,\prime},(\beta)u,\alpha^{\,\prime})\cdot
    (\alpha^{\,\prime},e,\alpha^{\,\prime})\cdot
    (\alpha^{\,\prime},(\beta)v,\beta^{\,\prime})=\\
    =&\,(\beta^{\,\prime},(\beta)u\cdot
    e\cdot(\beta)v,\alpha^{\,\prime})= (\beta^{\,\prime},
    (\beta)u\cdot(\beta)v\cdot e,\beta^{\,\prime})=\\
    =&\,(\beta^{\,\prime},f\cdot e,\beta^{\,\prime}),
\end{split}
\end{equation*}
and hence $e=f\cdot e=f$. Therefore
$(\beta)v\cdot(\beta)u=(\beta)u\cdot(\beta)v=e$,
$(\beta)v,(\beta)u\in H_e$, and hence $(\beta)v$ and $(\beta)u$
are inverse elements in the subgroup $H_e$. If
$(\gamma)\varphi=(\delta)\varphi$ for $\gamma,\delta\in
I_{\lambda_1}$ then
\begin{equation*}
    0_1\neq(\alpha^{\,\prime},e,(\gamma)\varphi)\cdot
    ((\delta)\varphi,e,\alpha^{\,\prime})=
    ((\alpha,1_S,\gamma))\sigma\cdot((\delta,1_S,\alpha))\sigma,
\end{equation*}
and since $\sigma$ is a non-trivial homomorphism, we have
$(\alpha,1_S,\gamma)\cdot(\delta,1_S,\alpha)\neq 0$ and hence
$\gamma=\delta$. Thus $\varphi\colon I_{\lambda_1}\rightarrow
I_{\lambda_2}$ is a one-to-one map.

Therefore for $s\in S\setminus I_h$ we have
\begin{equation*}
\begin{split}
    ((\gamma,s,\delta))\sigma=&\,
    ((\gamma,1_S,\alpha)\cdot(\alpha,s,\alpha)\cdot(\alpha,1_S,\delta))\sigma=\\
    =&\,((\gamma,1_S,\alpha))\sigma\cdot ((\alpha,s,\alpha))\sigma
    \cdot((\alpha,1_S,\delta))\sigma=\\
    =&\,((\gamma)\varphi,(\gamma)u,\alpha^{\,\prime})
    \cdot(\alpha^{\,\prime},(s)h,\alpha^{\,\prime})
    \cdot(\alpha^{\,\prime},((\delta)u)^{-1},(\delta)\varphi)=\\
    =&\,((\gamma)\varphi,(\gamma)u\cdot(s)h\cdot((\delta)u)^{-1},(\delta)\varphi).
\end{split}
\end{equation*}
Corollary~\ref{corollary2-9} implies that
$((\alpha,s,\beta))\sigma=0_2$ for all $s\in I_h$ and by
Proposition~\ref{proposition2-4} we conclude that
$(0_1)\sigma=0_2$.
\end{proof}

\begin{remark}\label{remark2-11}
We observe that if a semigroup $T$ satisfies one of the following
conditions:
\begin{itemize}
    \item[($i$)] $T^*$ is a cancellative monoid; or
    \item[($ii$)] $T$ is an inverse Clifford monoid,
\end{itemize}
then the second assertion of Theorem~\ref{theorem2-10} holds. Also,
Examples~\ref{example2-12} and \ref{example2-13} imply that this
statement is not true for inverse monoids with zeros and
completely regular monoids with zeros.
\end{remark}

\begin{example}\label{example2-12}
Let $T$ be the bicyclic semigroup $\mathscr{C}(p,q)=\langle p,
q\mid p\,q=1\rangle$ with adjoined zero. Then we can write every
element of the semigroup $\mathscr{C}(p,q)$ as $q^ip^j$ for some
$i,j=0,1,2,\ldots$~. We define a homomorphism $\sigma\colon
B_2\rightarrow B_2^0(\mathscr{C}(p,q))$ as follows:
\begin{align*}
    (0_1)&\sigma=0_2, & &{}\\
    ((1,1,1))&\sigma=(1,1,1), & ((1,1,2))&\sigma=(1,p,2),\\
    ((2,1,2))&\sigma=(2,qp,2), & ((2,1,1))&\sigma=(2,q,1).
\end{align*}
\end{example}

\begin{example}\label{example2-13}
Let $T$ be a $2\times 2$-rectangular band with adjoined unity
$1_T$ and zero $0_T$, i.~e. $T=\{(1,1),(1,2,),(2,1),(2,2), 0_T\}$.
We define a homomorphism $\sigma\colon B_2^0\rightarrow B_2(T)$ as
follows:
\begin{align*}
    (0_1)&\sigma=0_2, & &{}\\
    ((1,1,1))&\sigma=(1,(1,1),1), & ((1,1,2))&\sigma=(1,(1,2),2),\\
    ((2,1,2))&\sigma=(2,(2,2),2), & ((2,1,1))&\sigma=(2,(2,1),1).
\end{align*}
\end{example}

We observe that a composition of two non-trivial homomorphisms of
the Brandt $\lambda$-extensions of monoids with zeros may be the
trivial homomorphism. This observation follows from the next
example.

\begin{example}\label{example2-14}
Consider the set $E=\{a,b,c\}$ with the following semigroup
operation:
\begin{align*}
    a\cdot a=&\, a, & a\cdot b= &\, b\cdot a=b & \textrm{ and }& &
    a\cdot c=c\cdot a=b\cdot c=c\cdot b=c\cdot c=c.
\end{align*}
Then $E$ with this operation is a semilattice with zero $c$ and
unity $a$, and hence the conditions $(i)-(ii)$ of
Theorem~\ref{theorem2-10} hold for the monoid $E$. We define a
homomorphism $h\colon E\rightarrow E$ as follows: $(a)h=b$ and
$(b)h=(c)h=c$. Then for any non-empty set $I_\lambda$ the
homomorphism $\sigma\colon B_\lambda^0(E)\rightarrow
B_\lambda^0(E)$ defined by formulae
\begin{align*}
    ((\alpha,a,\beta))\sigma=&\,(\alpha,b,\beta), &
    ((\alpha,b,\beta))\sigma=&\,0, & \textrm{and} & & (0)\sigma=0,
\end{align*}
where $\alpha,\beta\in I_\lambda$ and $0$ is the zero of the
semigroup $B_\lambda^0(E)$, the composition
$\sigma\circ\sigma\colon B_\lambda^0(E)\rightarrow B_\lambda^0(E)$
is the trivial homomorphism.
\end{example}

Proposition~\ref{proposition1-1}$(i)$ yields simple sufficient
conditions that a composition of non-trivial homomorphisms of
the Brandt $\lambda^0$-extensions is a non-trivial homomorphism:

\begin{proposition}\label{proposition2-15}
Let $\lambda_1$, $\lambda_2$ and $\lambda_3$ be cardinals such
that $\lambda_1\leqslant\lambda_2\leqslant\lambda_3$ and let $S$,
$T$ and $R$ be monoids with zeros. Let $\sigma_1\colon
B_{\lambda_1}^0(S)\rightarrow B_{\lambda_2}^0(T)$ and
$\sigma_2\colon B_{\lambda_2}^0(T)\rightarrow B_{\lambda_3}^0(R)$
be non-trivial homomorphisms. If one of the following conditions
holds
\begin{itemize}
    \item[$(i)$] For some $\alpha\in I_{\lambda_1}$ the
    restriction $\sigma_1|_{S_{\alpha,\alpha}}\colon
    S_{\alpha,\alpha}\rightarrow
    (S_{\alpha,\alpha})\sigma_1\subset B_{\lambda_2}^0(T)$ of
    $\sigma_1$ is a monoid homomorphism; or
    \item[$(ii)$] $E(T)=\{1_T,0_T\}$,
\end{itemize}
then the composition $\sigma_1\circ\sigma_2\colon
B_{\lambda_1}^0(S)\rightarrow B_{\lambda_3}^0(R)$ is a non-trivial
homomorphism.
\end{proposition}

Since the semigroup of matrix units is congruence-free (cf. 
Theorem~1 of \cite{Gluskin1955}), we get the following
proposition:

\begin{proposition}\label{proposition2-16}
Let $\lambda_1$, $\lambda_2$ and $\lambda_3$ be cardinals such
that $\lambda_1\leqslant\lambda_2\leqslant\lambda_3$ and let $S$,
$T$ and $R$ be monoids with zeros. Let $\sigma_1 \colon
B_{\lambda_1}^0(S)\rightarrow B_{\lambda_2}^0(T)$ and
$\sigma_2\colon B_{\lambda_2}^0(T)\rightarrow B_{\lambda_3}^0(R)$
be non-trivial homomorphisms. Then the composition
$\sigma_1\circ\sigma_2\colon B_{\lambda_1}^0(S)\rightarrow
B_{\lambda_3}^0(R)$ is a non-trivial homomorphism if and only if
$((\alpha,1_S,\beta))\sigma_1\notin
\{(\alpha^{\,\prime},t,\beta^{\,\prime})\in B_{\lambda_2}^0(T)\mid
((\alpha^{\,\prime},t,\beta^{\,\prime}))\sigma_2=0_3\}$, for some
$\alpha,\beta\in I_{\lambda_1}$,
$\alpha^{\,\prime},\beta^{\,\prime}\in I_{\lambda_2}$.
\end{proposition}

\section{The category of the Brandt $\lambda^0$-extensions of
monoids with zeros}\label{sec3}

Let $S$ and $T$ be monoids with zeros. Let $\texttt{Hom}\,_0(S,T)$
be the set of all homomorphisms $\sigma\colon S\rightarrow T$ such
that $(0_S)\sigma=0_T$. We put
\begin{equation*}
    \mathbf{E}_1(S,T)=\{e\in E(T)\mid \hbox{there exists~} \sigma\in
    \texttt{Hom}\,_0(S,T) \hbox{~such that~} (1_S)\sigma=e\}
\end{equation*}
and define the family
\begin{equation*}
    \mathscr{H}_1(S,T)=\{ H(e)\mid e\in\mathbf{E}_1(S,T)\},
\end{equation*}
where we denote the maximal subgroup with the unity $e$ in the
semigroup $T$ by $H(e)$. Also by $\mathfrak{B}$ we denote the
class of monoids $S$ with zeros such that $S$ has
$\mathcal{B}^*$-property and every idempotent of $S$ lies in the
center of $S$.

We define a category $\mathscr{B}$ as follows:
\begin{itemize}
    \item[$(i)$] $\operatorname{\textbf{Ob}}(\mathscr{B})=
    \{(S,I)\mid S\in\mathfrak{B} \textrm{~and~} I \textrm{~is a
    non-empty set}\}$, and if $S$ is a trivial semigroup then we
    identify $(S,I)$ and $(S,J)$ for all non-empty sets $I$ and
    $J$;
    \item[$(ii)$] $\operatorname{\textbf{Mor}}(\mathscr{B})$
    consists of triples $(h,u,\varphi)\colon
    (S,I)\rightarrow(S^{\,\prime},I^{\,\prime})$, where
\begin{equation}\label{eq3-1}
\begin{split}
    & h\colon S\rightarrow S^{\,\prime} \textrm{~is a
      homomorphism such that~}
      h\in \texttt{Hom}\,_0(S,S^{\,\prime}),\\
    & u\colon I\rightarrow H(e) \textrm{~is a map~}, \textrm{~for~} H(e)\in
      \mathscr{H}_1(S,S^{\,\prime}),\\
    & \varphi\colon I\rightarrow I^{\,\prime} \textrm{~is an
    one-to-one map},
\end{split}
\end{equation}
with the composition
\begin{equation}\label{eq3-2}
    (h,u,\varphi)(h^{\,\prime},u^{\,\prime},\varphi^{\,\prime})=
    (hh^{\,\prime},[u,\varphi,h^{\,\prime},u^{\,\prime}],\varphi\varphi^{\,\prime}),
\end{equation}
where the map $[u,\varphi,h^{\,\prime},u^{\,\prime}]\colon
I\rightarrow H(e)$ is defined by the formula
\begin{equation}\label{eq3-3}
    (\alpha)[u,\varphi,h^{\,\prime},u^{\,\prime}]=
    ((\alpha)\varphi)u^{\,\prime}\cdot((\alpha)u)h^{\,\prime}
    \qquad \textrm{~for~} \alpha\in I.
\end{equation}
\end{itemize}
A straightforward verification shows that $\mathscr{B}$ is a
category with the identity morphism
$\varepsilon_{(S,I)}=(\operatorname{Id}_S,u_0,\operatorname{Id}_I)$
for any $(S,I)\in\operatorname{\textbf{Ob}}(\mathscr{B})$, where
$\operatorname{Id}_S\colon S\rightarrow S$ and
$\operatorname{Id}_I\colon I\rightarrow I$ are identity maps and
$(\alpha)u_0=1_S$ for all $\alpha\in I$.

We define the category $\mathscr{B}^*(\mathscr{S})$ as follows:
\begin{itemize}
    \item[$(i)$] $\operatorname{\textbf{Ob}}(\mathscr{B}^*(\mathscr{S}))$
    are all Brandt $\lambda^0$-extensions of monoids $S$ with zeros
    such that $S$ has the $\mathcal{B}^*$-property and every
    idempotent of $S$ lies in the center of $S$;
    \item[$(ii)$] $\operatorname{\textbf{Mor}}(\mathscr{B}^*(\mathscr{S}))$
    are homomorphisms of the Brandt $\lambda^0$-extensions
    of monoids $S$ with zeros such that $S$ has the
    $\mathcal{B}^*$-property and every idempotent of $S$ lies in
    the center of $S$.
\end{itemize}

For each
$(S,I_{\lambda_1})\in\operatorname{\textbf{Ob}}(\mathscr{B})$ with
non-trivial $S$, let
$\textbf{B}(S,I_{\lambda_1})=B_{\lambda_1}^0(S)$ be the Brandt
$\lambda^0$-extension of the semigroup $S$. For each
$(h,u,\varphi)\in\operatorname{\textbf{Mor}}(\mathscr{B})$ with a
non-trivial homomorphism $h$, where $(h,u,\varphi)\colon
(S,I_{\lambda_1})\rightarrow(T,I_{\lambda_2})$ and
$(T,I_{\lambda_2})\in\operatorname{\textbf{Ob}}(\mathscr{B})$, we
define a map $\textbf{B}{(h,u,\varphi)}\colon
\textbf{B}(S,I_{\lambda_1})=B_{\lambda_1}^0(S)\rightarrow
\textbf{B}(T,I_{\lambda_2})=B_{\lambda_2}^0(T)$ as follows:
\begin{equation}\label{functor_B}
    ((\alpha,s,\beta))[\textbf{B}{(h,u,\varphi)}]=
    \left\{%
\begin{array}{cl}
    ((\alpha)\varphi,(\alpha)u\cdot(s)h\cdot((\beta)u)^{-1},(\beta)\varphi),
    & \hbox{if}\quad s\notin S\setminus I_h ;\\
    0_2, & \hbox{if}\quad s\in I_h^*,\\
\end{array}%
\right.
\end{equation}
and $(0_1)[\textbf{B}{(h,u,\varphi)}]=0_2$, where $I_h=\{ s\in
S\mid (s)h=0_T\}$ is an ideal of $S$ and $0_1$ and $0_2$ are the
zeros of the semigroups $B_{\lambda_1}^0(S)$ and
$B_{\lambda_2}^0(T)$, respectively. For each
$(h,u,\varphi)\in\operatorname{\textbf{Mor}}(\mathscr{B})$ with a
trivial homomorphism $h$ we define a map
$\textbf{B}{(h,u,\varphi)}\colon
\textbf{B}(S,I_{\lambda_1})=B_{\lambda_1}^0(S)\rightarrow
\textbf{B}(T,I_{\lambda_2})=B_{\lambda_2}^0(T)$ as follows:
$(a)[\textbf{B}{(h,u,\varphi)}]=0_2$ for all $a\in
\textbf{B}(S,I_{\lambda_1})=B_{\lambda_1}^0(S)$. If $S$ is a
trivial semigroup then we define $\textbf{B}(S,I_{\lambda_1})$ to
be a trivial semigroup.

A functor $\textbf{F}$ from a category $\mathscr{C}$ into a
category $\mathscr{K}$ is called \emph{full} if for any
$a,b\in\operatorname{\textbf{Ob}}(\mathscr{C})$ and for any
$\mathscr{K}$-morphism $\alpha\colon \textbf{F}a\rightarrow
\textbf{F}b$ there exists a $\mathscr{C}$-morphism $\beta\colon
a\rightarrow b$ such that $\textbf{F}\beta=\alpha$, and
$\textbf{F}$ called \emph{representative} if for any
$a\in\operatorname{\textbf{Ob}}(\mathscr{K})$ there exists
$b\in\operatorname{\textbf{Ob}}(\mathscr{C})$ such that $a$ and
$\textbf{F}b$ are isomorphic.

\begin{theorem}\label{theorem3-1}
$\operatorname{\textbf{B}}$ is a full representative functor from
$\mathscr{B}$ into $\mathscr{B}^*(\mathscr{S})$.
\end{theorem}

\begin{proof}
For any $(S,I_{\lambda})
\in\operatorname{\textbf{Ob}}(\mathscr{B})$,
$\textbf{B}((S,I_{\lambda}))$ is the the Brandt $\lambda^0$-extension
of the monoid with zero $S$ by Proposition~\ref{proposition2-2},
and hence we have that $\textbf{B}(S,I_{\lambda})\in
\operatorname{\textbf{Ob}}(\mathscr{B}^*(\mathscr{S}))$. By
Theorem~\ref{theorem2-10} we have that for a
$\mathscr{B}$-morphism $(h,u,\varphi)\colon
(S,I_{\lambda_1})\rightarrow (T,I_{\lambda_2})$,
$\textbf{B}{(h,u,\varphi)}$ is a non-trivial homomorphism of
$\textbf{B}(S,I_{\lambda_1})$ into $\textbf{B}(T,I_{\lambda_2})$
in the case when $h$ is a non-trivial homomorphism. Obviously,
$\textbf{B}\varepsilon_{(S,I)}=
\textbf{B}(\operatorname{Id}_S,u_0,\operatorname{Id}_I)$ is the
identity automorphism of $\textbf{B}(S,I)$. Let
$(h,u,\varphi)\colon (S,I_{\lambda_1})\rightarrow
(T,I_{\lambda_2})$ and $(f,v,\psi)\colon
(T,I_{\lambda_2})\rightarrow (R,I_{\lambda_3})$ be
$\mathscr{B}$-morphisms with non-trivial $h$ and $f$. Then for any
$(\alpha,s,\beta)\in\textbf{B}(S,I_{\lambda_1})$ we get
\begin{equation*}
\begin{split}
    (\alpha,s,\beta)&\,[\textbf{B}(h,u,\varphi)]\big[\textbf{B}(f,v,\psi)\big]=\\
        =&\left\{%
\begin{array}{cl}
    \big(((\alpha)\varphi,(\alpha)u\cdot(s)h\cdot((\beta)u)^{-1},(\beta)\varphi)\big)
    \big[\textbf{B}(f,v,\psi)\big],
    & \hbox{if}\quad s\notin S\setminus I_h ;\\
    (0_2)\big[\textbf{B}(f,v,\psi)\big], & \hbox{if}\quad s\in I_h^*,\\
\end{array}%
\right.\\
\end{split}
\end{equation*}
\begin{equation*}
\begin{split}
 &\big(((\alpha)\varphi,(\alpha)u\cdot(s)h\cdot((\beta)u)^{-1},(\beta)\varphi)\big)
    \big[\textbf{B}(f,v,\psi)\big]=\\
    =&\left\{%
\begin{array}{cl}
    \!\!\!\!\big(((\alpha)\varphi)\psi{,}
    ((\alpha)\varphi)v{\cdot}\big((\alpha)u{\cdot}(s)h{\cdot}((\beta)u)^{-1}\big)f
    {\cdot}\big(((\beta)\psi)v\big)^{-1}\!{,}
    ((\beta)\varphi)\psi\big),
    & \!\!\!\hbox{if~} (\alpha)u{\cdot}(s)h{\cdot}((\beta)u)^{-1}{\notin} T{\setminus} I_f;\\
    (0_2)\big[\textbf{B}(f,v,\psi)\big], & \!\!\!\hbox{if~}
    (\alpha)u{\cdot}(s)h{\cdot}((\beta)u)^{-1}{\in} I_f^*,\\
\end{array}%
\right.\\
    =&\left\{%
\begin{array}{cl}
    \!\!\!\!\big(((\alpha)\varphi)\psi{,}
    \big(((\alpha)\varphi)v{\cdot}((\alpha)u)f\big)
    {\cdot}((s)h)f
    {\cdot}\big(((\beta)\varphi)v{\cdot}((\beta)u)f\big)^{-1}\!{,}
    ((\beta)\varphi)\psi\big),
    & \!\!\!\hbox{if~} (\alpha)u{\cdot}(s)h{\cdot}((\beta)u)^{-1}{\notin} T{\setminus} I_f;\\
    (0_2)\big[\textbf{B}(f,v,\psi)\big], & \!\!\!\hbox{if~}
    (\alpha)u{\cdot}(s)h{\cdot}((\beta)u)^{-1}{\in} I_f^*,\\
\end{array}%
\right.
\end{split}
\end{equation*}
and $(0_1)[\textbf{B}(h,u,\varphi)]\big[\textbf{B}(f,v,\psi)\big]=
(0_2)\big[\textbf{B}(f,v,\psi)\big]=0_3$. In the case when for at
least one of the $\mathscr{B}$-morphisms $(h,u,\varphi)\colon
(S,I_{\lambda_1})\rightarrow (T,I_{\lambda_2})$ and
$(f,v,\psi)\colon (T,I_{\lambda_2})\rightarrow (R,I_{\lambda_3})$,
either $h$ or $f$ is trivial, we have by
Proposition~\ref{proposition2-1} that
$(x)[\textbf{B}(h,u,\varphi)]\big[\textbf{B}(f,v,\psi)\big]=0_3$.
Therefore $\textbf{B}$ preserves the compositions of morphisms,
and hence $\textbf{B}$ is a functor from $\mathscr{B}$ into
$\mathscr{B}^*(\mathscr{S})$.

Theorem~\ref{theorem2-10} implies that the functor $\textbf{B}$ is
full and by the definition of the Brandt $\lambda^0$-extension we
conclude that the functor $\textbf{B}$ is representative.
\end{proof}

We define the first series of categories $\mathscr{BI}$,
$\mathscr{BIC}$, $\mathscr{BSL}$, $\mathscr{BS}_2$ and
$\mathscr{BG}$ as follows:
\begin{itemize}
    \item[$(i)$] $\operatorname{\textbf{Ob}}(\mathscr{BI})=
    \{(S,I)\mid S\in\mathfrak{B} \textrm{~is an inverse monoid~and~} I \textrm{~is a
    non-empty set}\}$, and if $S$ is a trivial semigroup then we
    identify $(S,I)$ and $(S,J)$ for all non-empty sets $I$ and
    $J$;
    \item[] $\operatorname{\textbf{Ob}}(\mathscr{BIC})=
    \{(S,I)\mid S\textrm{~is an inverse Clifford monoid~and~} I \textrm{~is a
    non-empty set}\}$, and if $S$ is a trivial semigroup then we
    identify $(S,I)$ and $(S,J)$ for all non-empty sets $I$ and
    $J$;
    \item[] $\operatorname{\textbf{Ob}}(\mathscr{BSL})=
    \{(S,I)\mid S\textrm{~is a semilattice with unity and zero~and~} I$ ~is a
    non-empty set$\}$, and if $S$ is a trivial semigroup then we
    identify $(S,I)$ and $(S,J)$ for all non-empty sets $I$ and
    $J$;
    \item[] $\operatorname{\textbf{Ob}}(\mathscr{BS}_2)=
    \{(S,I)\mid S$~is a monoid~with two idempotents, the zero and the unity and $I$ is a
    non-empty set$\}$;
    \item[] $\operatorname{\textbf{Ob}}(\mathscr{BG})=
    \{(S,I)\mid S$~is a group and $I$ is a non-empty set$\}$;
    \item[$(ii)$] $\operatorname{\textbf{Mor}}(\mathscr{BI})$,
    $\operatorname{\textbf{Mor}}(\mathscr{BIC})$,
    $\operatorname{\textbf{Mor}}(\mathscr{BSL})$,
    $\operatorname{\textbf{Mor}}(\mathscr{BS}_2)$ and
    $\operatorname{\textbf{Mor}}(\mathscr{BG})$
    consist of corresponding triples $(h,u,\varphi)\colon$
    $
    (S,I)\rightarrow(S^{\,\prime},I^{\,\prime})$, which satisfy
    condition (\ref{eq3-1}).
\end{itemize}
Obviously, $\mathscr{BI}$, $\mathscr{BIC}$, $\mathscr{BSL}$,
$\mathscr{BS}_2$ and $\mathscr{BG}$ are subcategories of the
category $\mathscr{B}$.

The second series of categories $\mathscr{B}^*(\mathscr{IS})$,
$\mathscr{B}^*(\mathscr{ICS})$, $\mathscr{B}^*(\mathscr{SL})$,
$\mathscr{B}^*(\mathscr{S}_2)$ and $\mathscr{B}^*(\mathscr{G})$ is
defined as follows:
\begin{itemize}
    \item[$(i)$] $\operatorname{\textbf{Ob}}(\mathscr{B}^*(\mathscr{IS}))$
    are all Brandt $\lambda^0$-extensions of inverse monoids $S$ with zeros
    such that $S$ has $\mathfrak{B}^*$-property and every
    idempotent of $S$ lies in the center of $S$;
    \item[] $\operatorname{\textbf{Ob}}(\mathscr{B}^*(\mathscr{ICS}))$
    are all Brandt $\lambda^0$-extensions of inverse Clifford monoids
    with zeros;
    \item[] $\operatorname{\textbf{Ob}}(\mathscr{B}^*(\mathscr{SL}))$
    are all Brandt $\lambda^0$-extensions of semilattices with zeros and
    identities;
    \item[] $\operatorname{\textbf{Ob}}(\mathscr{B}^*(\mathscr{S}_2))$
    are all Brandt $\lambda^0$-extensions of monoids~with two
    idempotents zeros and identities;
    \item[] $\operatorname{\textbf{Ob}}(\mathscr{B}^*(\mathscr{G}))$
    are all Brandt semigroups;
    \item[$(ii)$] $\operatorname{\textbf{Mor}}(\mathscr{B}^*(\mathscr{IS}))$,
    $\operatorname{\textbf{Mor}}(\mathscr{B}^*(\mathscr{ICS}))$,
    $\operatorname{\textbf{Mor}}(\mathscr{B}^*(\mathscr{SL}))$ and
    $\operatorname{\textbf{Mor}}(\mathscr{B}^*(\mathscr{S}_2))$
    are homomorphisms of the Brandt $\lambda^0$-extensions
    of monoids $S$ with zeros such that $S$ has the
    $\mathcal{B}^*$-property and every idempotent of $S$ lies in
    the center of $S$, inverse Clifford monoids with zeros, semilattices
    with zeros and identities, monoids~with two idempotents zeros and
    identities and $\operatorname{\textbf{Mor}}(\mathscr{B}^*(\mathscr{G}))$
    be non-trivial
    homomorphisms of Brandt semigroups.
\end{itemize}

The proof of the following proposition is similar to the proof of
Theorem~\ref{theorem3-1}.

\begin{proposition}\label{proposition3-2}
$\textbf{B}$ is a full representative functor from $\mathscr{BI}$
$[$resp., $\mathscr{BIC}$, $\mathscr{BSL}$, $\mathscr{BS}_2$ and
$\mathscr{BG}$$]$ into $\mathscr{B}^*(\mathscr{IS})$ $[$resp.,
$\mathscr{B}^*(\mathscr{ICS})$, $\mathscr{B}^*(\mathscr{SL})$,
$\mathscr{B}^*(\mathscr{S}_2)$ and
$\mathscr{B}^*(\mathscr{G})$$]$.
\end{proposition}

\begin{proposition}\label{proposition3-3}
Let $(h,u,\varphi)\colon (S,I_{\lambda_1})\rightarrow
(T,I_{\lambda_2})$ and $(f,v,\psi)\colon
(S,I_{\lambda_1})\rightarrow (T,I_{\lambda_2})$ be
$\mathscr{BSL}$-morphisms. Then
$\textbf{B}(h,u,\varphi)=\textbf{B}(f,v,\psi)$ if and only if
$h=f$, $u=v$ and $\varphi=\psi$.
\end{proposition}

\begin{proof}
By definition of the functor $\textbf{B}$ we have
$\textbf{B}(h,u,\varphi)=\textbf{B}(f,v,\psi)$ if and only if
\begin{equation*}
    ((\alpha)\varphi,(\alpha)u\cdot(s)h\cdot((\beta)u)^{-1},(\beta)\varphi)=
    ((\alpha)\psi,(\alpha)v\cdot(s)f\cdot((\beta)v)^{-1},(\beta)\psi),
\end{equation*}
for $(\alpha,s,\beta)\in\textbf{B}(S,I_{\lambda_1})$. Then
$\varphi=\psi$ and since for semilattices $S$ we have
$(\alpha)u=((\alpha)u)^{-1}=(1_S)h$ and
$(\alpha)v=((\alpha)v)^{-1}=(1_S)f$, we get that $h=f$.
\end{proof}

\begin{remark}\label{remark3-4}
The definition of
the
functor $\textbf{B}$ implies that $\textbf{B}$
is one-to-one on objects of the category $\mathscr{BI}$ [resp.,
$\mathscr{BIC}$, $\mathscr{BSL}$, $\mathscr{BS}_2$ and
$\mathscr{BG}$].
Also, Proposition~\ref{proposition3-3} implies
that the functor $\textbf{B}$ is one-to-one on morphisms of the
category $\mathscr{BSL}$, but Proposition~II.3.9 of
\cite{Petrich1984} implies that
the
functor $\textbf{B}$ is not
one-to-one on morphisms of the category $\mathscr{BI}$ [resp.,
$\mathscr{BIC}$, $\mathscr{BS}_2$ and $\mathscr{BG}$].
\end{remark}

Therefore Theorem~\ref{theorem3-1} and
Proposition~\ref{proposition3-3} imply:

\begin{corollary}\label{corollary3-5}
The categories $\mathscr{BSL}$ and $\mathscr{B}^*(\mathscr{SL})$
are isomorphic.
\end{corollary}


\section{Topological Brandt $\lambda^0$-extensions of topological
monoids with zero}

A topological space $S$ which
is algebraically
a
semigroup with a
jointly continuous semigroup operation is called a {\em
topological semigroup}. A {\em topological inverse semigroup} is a
topological semigroup $S$ that is algebraically an inverse
semigroup with continuous inversion. If $\tau$ is a topology on a
(inverse) semigroup $S$ such that $(S,\tau)$ is a topological
(inverse) semigroup, then $\tau$ is called a ({\em inverse}) {\em
semigroup topology} on $S$.

In this section we shall
follow the terminology
of~\cite{CarruthHildebrantKoch} and \cite{Engelking1989}.

\begin{definition}[\cite{GutikPavlyk2006}]\label{definition4-0}
Let $\mathscr{S}$ be some class of topological monoids with zero.
Let $\lambda$ be any cardinal $\geqslant 1$, and
$(S,\tau)\in\mathscr{S}$. Let $\tau_{B}$ be a topology on
$B^0_{\lambda}(S)$ such that:
\begin{itemize}
  \item[a)] $\left(B^0_{\lambda}(S), \tau_{B}\right)\in\mathscr{S};$ 
  and
  \item[b)] $\tau_{B}|_{S_{\alpha,\alpha}}=\tau$ for some
$\alpha\in I_{\lambda}$.
\end{itemize}
Then $\left(B^0_{\lambda}(S), \tau_{B}\right)$ is called the
\emph{topological Brandt $\lambda^0$-extension of $(S, \tau)$ in}
$\mathscr{S}$. If $\mathscr{S}$ coincides with the class of all
topological semigroups, then $\left(B^0_{\lambda}(S),
\tau_{B}\right)$ is called the
\emph{topological Brandt
$\lambda^0$-extension of} $(S, \tau)$.
\end{definition}

Results of Section~2 of 
\cite{GutikPavlyk2006} imply
that
for any infinite cardinal $\lambda$
and every non-trivial
topological semigroup $S$, there are many topological Brandt
$\lambda^0$-extensions of $S$, and for any topological inverse
semigroup $T$, there are many topological Brandt
$\lambda^0$-extensions of $T$ in the class of topological inverse
semigroups. Moreover, for any
infinite cardinal $\lambda$ on the
Brandt $\lambda^0$-extension of two-element monoid with zero
(i.~e., on the infinite semigroup of $I_\lambda\times
I_\lambda$-units) there exist many semigroup and inverse semigroup
topologies (cf.  \cite{GutikPavlyk2005}). 

These observations imply
that for infinite cardinals $\lambda$ there
does not exist a proposition
for topological semigroups similar to Theorem~\ref{theorem3-1} and
Propositions~\ref{proposition3-2} and \ref{proposition3-3}. In
this section we prove such statements for any finite non-zero
cardinals.

\begin{proposition}\label{proposition4-3}
Let $\lambda$ be any finite non-zero cardinal. Let be $(S,\tau)$ a
topological semigroup and $\tau_{B}$ a topology on
$B^0_{\lambda}(S)$ such that $\left(B^0_{\lambda}(S),
\tau_{B}\right)$ is a topological semigroup and
$\tau_{B}|_{S_{\alpha,\alpha}}=\tau$ for some $\alpha\in
I_{\lambda}$. Then the following assertions hold:
\begin{itemize}
    \item[$(i)$] If a non-empty subset $A\not\ni 0_S$ of $S$ is
    open in $S$, then so is $A_{\beta,\gamma}$ in
    $\left(B^0_{\lambda}(S),\tau_{B}\right)$ for any
    $\beta,\gamma\in I_{\lambda}$;

    \item[$(ii)$] If a non-empty subset $A\ni 0_S$ of $S$ is
    open in $S$, then so is $\bigcup_{\beta,\gamma\in
    I_{\lambda}}A_{\beta,\gamma}$ in
    $\left(B^0_{\lambda}(S),\tau_{B}\right)$;

    \item[$(iii)$] If a non-empty subset $A\not\ni 0_S$ of $S$ is
    closed in $S$, then so is $A_{\beta,\gamma}$ in
    $\left(B^0_{\lambda}(S),\tau_{B}\right)$ for any
    $\beta,\gamma\in I_{\lambda}$;

    \item[$(iv)$] If a non-empty subset $A\ni 0_S$ of $S$ is
    closed in $S$, then so is $\bigcup_{\beta,\gamma\in
    I_{\lambda}}A_{\beta,\gamma}$ in
    $\left(B^0_{\lambda}(S),\tau_{B}\right)$;

    \item[$(v)$] If $x$ is a non-zero element of $S$ and
    $\mathscr{B}_x$ is a base of the topology $\tau$ at $x$, then
    the family $\mathscr{B}_{(\beta,x,\gamma)}=
    \{U_{\beta,\gamma}\mid U\in\mathscr{B}_x\}$ is a base of the
    topology $\tau_B$ at the point $(\beta,x,\gamma)\in
    B^0_{\lambda}(S)$ for any $\beta,\gamma\in I_{\lambda}$;

    \item[$(vi)$] If $\mathscr{B}_{0_S}$ is a base of the topology
    $\tau$ at zero $0_S$ of $S$, then the family
    $\mathscr{B}_0=\{\bigcup_{\beta,\gamma\in
    I_{\lambda}}U_{\beta,\gamma}\mid U\in\mathscr{B}_{0_S}\}$ is a
    base of the topology $\tau_B$ at zero $0$ of the semigroup
    $B^0_{\lambda}(S)$.
\end{itemize}
\end{proposition}

\begin{proof}
$(i)$ Let $W\not\ni 0$ be an open set in
$\left(B^0_{\lambda}(S),\tau_{B}\right)$ such that $W\cap
S_{\alpha,\alpha}\in\tau_{B}|_{S_{\alpha,\alpha}}$. Suppose that
$W\nsubseteq S_{\alpha,\alpha}^*$. We fix $(\alpha,x,\alpha)\in
W$. Since $(\alpha,1_S,\alpha)\cdot(\alpha,x,\alpha)\cdot
(\alpha,1_s,\alpha)=(\alpha,x,\alpha)$, there exists an open
neighbourhood $U$ of the point $(\alpha,x,\alpha)$ such that
$U\subseteq W$ and $(\alpha,1_S,\alpha)\cdot U\cdot
(\alpha,1_s,\alpha)\subseteq W$. If $U\nsubseteq
S_{\alpha,\alpha}^*$, then $0\in(\alpha,1_S,\alpha)\cdot U\cdot
(\alpha,1_s,\alpha)\subseteq W$, a contradiction. Therefore for
any $(\alpha,x,\alpha)\in W$ there exists an open neighbourhood of
$(\alpha,x,\alpha)$ such that $U\subseteq S_{\alpha,\alpha}^*$,
and hence $W\cap S_{\alpha,\alpha}^*$ is an open subset in
$\left(B^0_{\lambda}(S),\tau_{B}\right)$.

By Definition~\ref{definition4-0} the set $A_{\alpha,\alpha}$ is
open for some $\alpha\in I_{\lambda}$. Since the map
$\varphi^{\alpha\alpha}_{\gamma\delta}\colon B^0_{\lambda}(S)
\rightarrow B^0_{\lambda}(S)$ defined by the formula
$(x)\varphi^{\alpha\alpha}_{\gamma\delta}=(\alpha,1_S,\gamma)\cdot
x\cdot (\delta,1_S,\alpha)$ is continuous, we get that the set
$A_{\gamma,\delta}=
(A_{\alpha,\alpha})\big(\varphi^{\alpha\alpha}_{\gamma\delta}\big)^{-1}$
is open in $\left(B^0_{\lambda}(S),\tau_{B}\right)$ for any
$\gamma,\delta\in I_{\lambda}$.

Let $A\ni 0$ be an open subset in $S$ and $W$ an open subset in
$\left(B^0_{\lambda}(S),\tau_{B}\right)$ such that $W\cap
S_{\alpha,\alpha}=A_{\alpha,\alpha}$ for some $\alpha\in
I_\lambda$. Since the map $\varphi^{\alpha\alpha}_{\gamma\delta}$
is continuous for any $\alpha,\gamma,\delta\in I_\lambda$, the set
\begin{equation*}
    \widetilde{A}_{\gamma,\delta}=
    \bigcup_{(\iota,\kappa)\in(I_\lambda\times
    I_\lambda)\setminus(\gamma,\delta)}S_{\iota,\kappa} \cup
    A_{\gamma,\delta}
\end{equation*}
is an open subset in $\left(B^0_{\lambda}(S),\tau_{B}\right)$.
Then since the set $I_\lambda$ is finite, we have that
$\displaystyle\bigcup_{\alpha,\beta\in
I_\lambda}{A_{\alpha,\beta}}=\bigcap_{\gamma,\delta\in
I_\lambda}\widetilde{A}_{\gamma,\delta}$ and this implies
statement $(ii)$.

Statements $(iii)-(vi)$ follow from $(i)$ and $(ii)$.
\end{proof}

\begin{remark}\label{remark4-4}
Note
that the statements $(i)$, $(iii)$, $(iv)$ and $(v)$ of
Proposition~\ref{proposition4-3} hold for any infinite cardinal
$\lambda$.
However,
Example~21 and Proposition~25 of
\cite{GutikPavlyk2005} imply that the statements $(ii)$ and $(vi)$
are
false
for any infinite cardinal $\lambda$.
\end{remark}

We shall
need the following lemma from~\cite{GutikPavlyk2006}:

\begin{lemma}[{\cite[Lemma~1]{GutikPavlyk2006}}]\label{lemma4-1}
Let $\lambda\geqslant 2$ be any cardinal and $B^0_{\lambda}(S)$ 
the topological Brandt $\lambda^0$-extension of a topological monoid
$S$ with zero. Let $T$ be a topological semigroup and $h\colon
B^0_{\lambda}(S)\to T$ be a continuous homomorphism. Then the sets
$(A_{\alpha\beta})h$ and $(A_{\gamma\delta})h$ are homeomorphic in
$T$ for all $\alpha, \beta, \gamma, \delta\in I_{\lambda}$, and
all $A\subseteq S$.
\end{lemma}

\begin{proof}
If $h$ is an annihilating homomorphism, then the statement of the
lemma is trivial.

Otherwise, we fix arbitrary $\alpha, \beta, \gamma, \delta\in
I_{\lambda}$ and define the maps
$\varphi^{\gamma\delta}_{\alpha\beta}\colon T\to T$ and
$\varphi^{\alpha\beta}_{\gamma\delta}\colon T\to T$ by the
formulae
\begin{equation*}
(s)\varphi^{\gamma\delta}_{\alpha\beta}=((\gamma, 1,
\alpha))h\cdot s\cdot ((\beta, 1, \delta))h
 \qquad \text{~and~} \qquad
 (s)\varphi_{\gamma\delta}^{\alpha\beta}=((\alpha, 1,
\gamma))h\cdot s\cdot ((\delta, 1, \beta))h,
\end{equation*}
$s\in T$. Obviously,
\begin{equation*}
\left(\!\big(\!\left((\alpha,x,
\beta)\right)h\big)\varphi^{\gamma\delta}_{\alpha\beta}\right)
\varphi_{\gamma\delta}^{\alpha\beta}= \left((\alpha, x,
\beta)\right)h
 \qquad \text{~and~} \qquad
\left(\!\big(\!\left((\gamma, x, \delta)\right)h
\big)\varphi_{\gamma\delta}^{\alpha\beta}
\right)\varphi^{\gamma\delta}_{\alpha\beta}= \left((\gamma, x,
\delta)\right)h,
\end{equation*}
for all  $\alpha, \beta, \gamma, \delta\in I_{\lambda}$, $x\in
S^1$, and hence
$\varphi^{\gamma\delta}_{\alpha\beta}\mid_{A_{\alpha\beta}}=
(\varphi_{\gamma\delta}^{\alpha\beta})^{-1}\mid_{A_{\alpha\beta}}$.
Since the maps $\varphi^{\gamma\delta}_{\alpha\beta}$ and
$\varphi_{\gamma\delta}^{\alpha\beta}$ are continuous on $T$, the
map
$\varphi^{\gamma\delta}_{\alpha\beta}\mid_{h(A_{\alpha\beta})}\colon
h(A_{\alpha\beta})\to h(A_{\gamma\delta})$ is a homeomorphism.
\end{proof}

\begin{proposition}\label{proposition4-4}
Let $\lambda\geqslant 1$ be any cardinal and $B^0_{\lambda}(S)$ 
the topological Brandt $\lambda^0$-extension of a topological
(inverse) monoid $S$ with zero in the class of topological
(inverse) semigroups $\mathfrak{S}$. Let $T\in\mathfrak{S}$ and
$h\colon B^0_{\lambda}(S)\to T$ be a continuous homomorphism. Then
the image $(B^0_{\lambda}(S))h$ is the topological Brandt
$\lambda^0$-extension of some monoid $M\in\mathfrak{S}$ with zero
in the class $\mathfrak{S}$.
\end{proposition}

\begin{proof}
Proposition~\ref{proposition2-2} implies the algebraic part of the
proposition. Since a subsemigroup of a topological semigroup is a
topological semigroup, Lemma~\ref{lemma4-1} implies that
$(B^0_{\lambda}(S))h$ is the topological Brandt
$\lambda^0$-extension of $(S_{\alpha,\alpha})h$ for some $a\in
I_\lambda$. Also if $S$ and $T$ are topological inverse
semigroups, then by Proposition~II.2
of~\cite{EberhartSelden1969} the image $(B^0_{\lambda}(S))h$ is
a topological inverse semigroup.
\end{proof}

Corollary~\ref{corollary2-3} and Proposition~\ref{proposition4-4}
imply:

\begin{corollary}\label{corollary4-5}
Let $B_\lambda(G)$ be a topological (inverse) Brandt semigroup.
Let $T$ be a topological (inverse) semigroup and $h\colon
B_{\lambda}(G)\to T$ be a continuous homomorphism. Then the image
$(B_{\lambda}(G))h$ is a topological (inverse) Brandt semigroup.
\end{corollary}

Proposition~\ref{proposition4-3} and Lemma~\ref{lemma4-1} imply
the following:

\begin{lemma}\label{lemma4-5a}
For any topological monoid $(S,\tau)$ with zero and for any finite
cardinal $\lambda\geqslant 1$ there exists a unique topological
Brandt $\lambda^0$-extension $\left(B^0_{\lambda}(S),
\tau_{B}\right)$ and the topology $\tau_{B}$ generated by the base
$\mathscr{B}_B=\bigcup\{\mathscr{B}_B(t)\mid t\in
B^0_{\lambda}(S)\}$, where:
\begin{itemize}
    \item[$(i)$] $\mathscr{B}_B(t)=\{(U(s))_{\alpha,\beta}
    \setminus\{ 0_S\}\mid U(s)\in\mathscr{B}_S(s)\}$, when
    $t=(\alpha,s,\beta)$ is a non-zero element of
    $B^0_{\lambda}(S)$, $\alpha,\beta\in I_\lambda$;
    \item[$(ii)$] $\mathscr{B}_B(0)=\{\bigcup_{\alpha,\beta\in
    I_\lambda}
    (U(0_S))_{\alpha,\beta}\mid U(0_S)\in\mathscr{B}_S(0_S)\}$,
    when $0$ is the zero of $B^0_{\lambda}(S)$,
\end{itemize}
and $\mathscr{B}_S(s)$ is a base of the topology $\tau$ at the
point $s\in S$.

Moreover, if $\lambda\geqslant 1$ is any finite cardinal then a
topological monoid $(S,\tau)$ with zero is a topological inverse
semigroup if and only if $\left(B^0_{\lambda}(S), \tau_{B}\right)$
is the topological Brandt $\lambda^0$-extension of $(S,\tau)$ in
the class of topological inverse semigroups.
\end{lemma}

The topological Brandt $\lambda^0$-extension
$\left(B^0_{\lambda}(S), \tau_{B}\right)$ is called \emph{compact}
(resp., \emph{countably compact}) if the topological space
$\left(B^0_{\lambda}(S), \tau_{B}\right)$ is compact (resp.,
countably compact).

Propositions~\ref{proposition4-5b} and~\ref{proposition4-5c}
describe the structures of compact Brandt $\lambda^0$-extensions
and countably compact Brandt $\lambda^0$-extensions in the class
of topological inverse semigroups.

\begin{proposition}\label{proposition4-5b}
A topological Brandt $\lambda^0$-extension $B^0_{\lambda}(S)$ of a
topological monoid $(S,\tau)$ with zero is compact if and only if
the cardinal $\lambda\geqslant 1$ is finite and $(S,\tau)$ is a
compact topological semigroup. Moreover, for any compact
topological monoid $(S,\tau)$ with zero and for any finite
cardinal $\lambda\geqslant 1$ there exists a unique compact
topological Brandt $\lambda^0$-extension $\left(B^0_{\lambda}(S),
\tau_{B}\right)$ and the topology $\tau_{B}$ generated by the base
$\mathscr{B}_B=\bigcup\{\mathscr{B}_B(t)\mid t\in
B^0_{\lambda}(S)\}$, where:
\begin{itemize}
    \item[$(i)$] $\mathscr{B}_B(t)=\{(U(s))_{\alpha,\beta}
    \setminus\{ 0_S\}\mid U(s)\in\mathscr{B}_S(s)\}$, when
    $t=(\alpha,s,\beta)$ is a non-zero element of
    $B^0_{\lambda}(S)$, $\alpha,\beta\in I_\lambda$;
    \item[$(ii)$] $\mathscr{B}_B(0)=\{\bigcup_{\alpha,\beta\in
    I_\lambda}
    (U(0_S))_{\alpha,\beta}\mid U(0_S)\in\mathscr{B}_S(0_S)\}$,
    when $0$ is the zero of $B^0_{\lambda}(S)$,
\end{itemize}
and $\mathscr{B}_S(s)$ is a base of the topology $\tau$ at the
point $s\in S$.
\end{proposition}

\begin{proof}
Since by Theorem~10 of \cite{GutikPavlyk2005}, the infinite
semigroup of matrix units does not embed into a compact
topological semigroup, the compactness of the topological Brandt
$\lambda^0$-extension $\left(B^0_{\lambda}(S), \tau_{B}\right)$ of
a topological semigroup $(S,\tau)$ implies that the cardinal
$\lambda$ is finite. Then by Theorem~1.7(e)
of~\cite[Vol.~1]{CarruthHildebrantKoch},
$(\alpha,1_S,\alpha)B^0_{\lambda}(S)(\alpha,1_S,\alpha)=S_{\alpha,\alpha}$
is a compact semigroup for any $\alpha\in I_\lambda$, and hence
$(S,\tau)$ is a compact topological semigroup. The converse
follows from Lemma~\ref{lemma4-1} and the assertion that the
finite union of compact spaces is a compact space.

Lemma~\ref{lemma4-5a} implies the last assertion of the
proposition.
\end{proof}

\begin{proposition}\label{proposition4-5c}
The topological Brandt $\lambda^0$-extension $B^0_{\lambda}(S)$ of a
topological monoid $(S,\tau)$ with zero in the class of
topological inverse semigroups is countably compact if and only if
the cardinal $\lambda\geqslant 1$ is finite and $(S,\tau)$ is a
countably compact topological inverse semigroup. Moreover, for any
countably compact topological monoid $(S,\tau)$ with zero and for
any finite cardinal $\lambda\geqslant 1$ there exists a unique
compact topological Brandt $\lambda^0$-extension
$\left(B^0_{\lambda}(S), \tau_{B}\right)$ in the class of
topological inverse semigroups and the topology $\tau_{B}$
generated by the base $\mathscr{B}_B=\bigcup\{\mathscr{B}_B(t)\mid
t\in B^0_{\lambda}(S)\}$, where:
\begin{itemize}
    \item[$(i)$] $\mathscr{B}_B(t)=\{(U(s))_{\alpha,\beta}
    \setminus\{ 0_S\}\mid U(s)\in\mathscr{B}_S(s)\}$, when
    $t=(\alpha,s,\beta)$ is a non-zero element of
    $B^0_{\lambda}(S)$, $\alpha,\beta\in I_\lambda$; and
    \item[$(ii)$] $\mathscr{B}_B(0)=\{\bigcup_{\alpha,\beta\in
    I_\lambda}
    (U(0_S))_{\alpha,\beta}\mid U(0_S)\in\mathscr{B}_S(0_S)\}$,
    when $0$ is the zero of $B^0_{\lambda}(S)$,
\end{itemize}
and $\mathscr{B}_S(s)$ is a base of the topology $\tau$ at the
point $s\in S$.
\end{proposition}

\begin{proof}
By Theorem~14
of
\cite{GutikPavlyk2005}, the semigroup of
$I_\lambda\times I_\lambda$-matrix units is a closed subsemigroup
of any topological inverse semigroup $T$ which contains it.
By
Theorem~6
of
\cite{GutikPavlyk2005}, on the infinite semigroup of matrix
units there exists no countably compact inverse semigroup
topology. Therefore $\lambda$ is a finite cardinal, hence by
Theorem~1.7(e) of~\cite[Vol.~1]{CarruthHildebrantKoch} and
Theorem~3.10.4
of
\cite{Engelking1989},
$(\alpha,1_S,\alpha)B^0_{\lambda}(S)(\alpha,1_S,\alpha)=S_{\alpha,\alpha}$
is a countably compact topological semigroup for any $\alpha\in
I_\lambda$, and thus $(S,\tau)$ is a countably compact topological
semigroup. The converse follows from Lemma~\ref{lemma4-1} and the
assertion that the finite union of countably compact spaces is a
countable compact space.

The last assertion of the proposition follows from
Lemma~\ref{lemma4-5a}.
\end{proof}

\begin{theorem}\label{theorem4-6}
Let $\lambda_1$ and $\lambda_2$ be any finite cardinals such that
$\lambda_2\geqslant\lambda_1\geqslant 1$. Let $B_{\lambda_1}^0(S)$
and $B_{\lambda_2}^0(T)$ be topological Brandt $\lambda_1^0$- and
$\lambda_2^0$-extensions of topological monoids $S$ and $T$ with
zero, respectively. Let $h\colon S\rightarrow T$ be a continuous
homomorphism such that $(0_S)h=0_T$ and $\varphi\colon
I_{\lambda_1}\rightarrow I_{\lambda_2}$ an one-to-one map. Let $e$
be a non-zero idempotent of $T$, $H_e$ a maximal subgroup of $T$
with unity $e$ and $u\colon I_{\lambda_1}\rightarrow H_e$ a map.
Then $I_h=\{ s\in S\mid (s)h=0_T\}$ is a closed ideal of $S$ and
the map $\sigma\colon B_{\lambda_1}^0(S)\rightarrow
B_{\lambda_2}^0(T)$ defined by the formulae
\begin{equation*}
    ((\alpha,s,\beta))\sigma=
    \left\{%
\begin{array}{cl}
    ((\alpha)\varphi,(\alpha)u\cdot(s)h\cdot((\beta)u)^{-1},(\beta)\varphi),
    & \hbox{if}\quad s\notin S\setminus I_h ;\\
    0_2, & \hbox{if}\quad s\in I_h^*,\\
\end{array}%
\right.
\end{equation*}
and $(0_1)\sigma=0_2$, is a non-trivial continuous homomorphism
from $B_{\lambda_1}^0(S)$ into $B_{\lambda_2}^0(T)$. Moreover if
for the semigroup $T$ the following conditions hold:
\begin{itemize}
    \item[($i$)] Every idempotent of $T$ lies in the center of
    $T$; and
    \item[($ii$)] $T$ has $\mathcal{B}_{\lambda_1}^*$-property,
\end{itemize}
then every non-trivial continuous homomorphism from
$B_{\lambda_1}^0(S)$ into $B_{\lambda_2}^0(T)$ can be so
constructed.
\end{theorem}

\begin{proof}
The algebraic part of the theorem follows from
Theorem~\ref{theorem3-1}.

Since the homomorphism $h$ is continuous, $I_h=(0_T)h^{-1}$ is a
closed ideal of the topological semigroup $S$.

Further we shall show that the homomorphism $\sigma$ is continuous
whenever is also $h$. We consider the following cases:
\begin{itemize}
    \item[$(i)$] $(0_1)\sigma=0_2$;
    \item[$(ii)$] $((\alpha, s,\beta))\sigma=0_2$, i.~e. $s\in
    I_h$; and
    \item[$(iii)$] $((\alpha, s,\beta))\sigma=
    ((\alpha)\varphi,(\alpha)u\cdot(s)h\cdot((\beta)u)^{-1},(\beta)\varphi)$,
\end{itemize}
where
$(\alpha, s,\beta)$ is any non-zero element of the semigroup
$B_{\lambda_1}^0(S)$.

Without loss of generality we may
assume that $\varphi\colon
I_{\lambda_1}\rightarrow I_{\lambda_2}$ is a bijection. Moreover,
for the simplification of the proof we can assume that
$(\alpha)\varphi=\alpha$ for all $\alpha\in I_{\lambda_1}$.

Consider case $(i)$. Let $U(0_2)=\bigcup_{\alpha,\beta\in
I_{\lambda_2}} (U(0_T))_{\alpha,\beta}$ be any open basic
neighbourhood of the zero $0_2$ in $B_{\lambda_2}^0(T)$. Since
left and right translations in $T$ and the homomorphism $h\colon
S\rightarrow T$ are continuous maps, there exists
for any
$(\alpha)u,((\beta)u)^{-1}\in B_{\lambda_1}^0(S)$,  an
open neighbourhood $V^{\alpha,\beta}(0_S)$ in $S$ such that
$(\alpha)u\cdot (V^{\alpha,\beta}(0_S))h\cdot
((\beta)u)^{-1}\subseteq U(0_T)$. We put
$V(0_S)=\bigcap_{\alpha,\beta\in
I_{\lambda_1}}V^{\alpha,\beta}(0_S)$ and
$V(0_1)=\bigcup_{\alpha,\beta\in I_{\lambda_1}}
(V(0_S))_{\alpha,\beta}$. Then $(V(0_1))\sigma\subseteq U(0_2)$.

In case $(ii)$ we have that $(s)h=0_T$. Let
$U(0_2)=\bigcup_{\alpha,\beta\in I_{\lambda_2}}
(U(0_T))_{\alpha,\beta}$ be any basic open neighbourhood of the
zero $0_2$ in $B_{\lambda_2}^0(T)$. Since left and right
translations in $T$ and the homomorphism $h\colon S\rightarrow T$
are continuous maps, for the open neighbourhood $U(0_T)$ of the
zero $0_T$ there exists an open neighbourhood $V(s)$ of $s$ in $S$
such that $(\alpha)u\cdot (V(s))h\cdot ((\beta)u)^{-1}\subseteq
U(0_T)$. Therefore we have that
$((V(s))_{\alpha,\beta})\sigma\subseteq U(0_2)$.

Next we consider case $(iii)$. Let $U_{\alpha,\beta}$ be any basic
open neighbourhood of the element $((\alpha, s,\beta))\sigma=
(\alpha,(\alpha)u\cdot(s)h\cdot((\beta)u)^{-1},\beta)$ in
$B_{\lambda_2}^0(T)$. Since left and right translations in the
semigroup $T$ and the homomorphism $h\colon S\rightarrow T$ are
continuous maps, there exists an open neighbourhood $V(s)$ of the
point $s$ in $S$ such that $(\alpha)u\cdot (V(s))h\cdot
((\beta)u)^{-1}\subseteq U$ and hence we get
$((V(s))_{\alpha,\beta})\sigma\subseteq U_{\alpha,\beta}$.

Since left and right translations in the topological semigroup
$B_{\lambda_2}^0(T)$ are continuous and any restriction of a
continuous map is a continuous map, the continuity of the
homomorphism $\sigma\colon B_{\lambda_1}^0(S)\rightarrow
B_{\lambda_2}^0(T)$ implies the continuity of $h$.
\end{proof}

Note
that the statements of Theorem~\ref{theorem4-6}
are false
for the
topological Brandt $\lambda^0$-extensions when $\lambda$
is
an
infinite cardinal. This follows from the next example:

\begin{example}\label{example4-9}
Let $\lambda$ be an infinite cardinal. On $B_\lambda$ we define a
topology $\tau_{mi}$ as follows: all non-zero elements of
$B_\lambda$ are isolated points and the family
\begin{equation*}
    \mathscr{B}(0)=\{ V_{\alpha_1}\cap\cdots\cap V_{\alpha_i}\cap
    H_{\beta_1}\cap\cdots\cap H_{\beta_j}\mid
    \alpha_1,\ldots,\alpha_i,\beta_1,\ldots,\beta_j\in I_\lambda,
    i,j\in\mathbb{N}\},
\end{equation*}
where $V_\gamma=B_\lambda\setminus\{(\gamma,\delta)\mid\delta\in
I_\lambda\}$ and
$H_\nu=B_\lambda\setminus\{(\delta,\nu)\mid\delta\in I_\lambda\}$,
$\gamma,\nu\in I_\lambda$, determined a base of the topology
$\tau_{mi}$ at zero of $B_\lambda$~(cf.  \cite{GutikPavlyk2005}).
By Proposition~25
of
\cite{GutikPavlyk2005}, $(B_\lambda,\tau_{mi})$
is a topological inverse semigroup.

Let $\mathfrak{d}$ be the discrete topology on $B_\lambda$. Then
the identity map $\sigma\colon
(B_\lambda,\tau_{mi})\rightarrow(B_\lambda,\mathfrak{d})$ is not
continuous, but the maps $h$, $\varphi$ and $u$ are as
requested 
in the statements of Theorem~\ref{theorem4-6}.
\end{example}

Similarly to Section~\ref{sec3},
we define new categories of
topological semigroups and pairs of finite sets and topological
semigroups.

Let $S$ and $T$ be topological monoids with zeros. Let be
$\texttt{CHom}\,_0(S,T)$ be the
set of all continuous homomorphisms
$\sigma\colon S\rightarrow T$ such that $(0_S)\sigma=0_T$. We put
\begin{equation*}
    \mathbf{E}^{\textit{top}}_1(S,T)=\{e\in E(T)\mid \hbox{there exists~}
    \sigma\in
    \texttt{CHom}\,_0(S,T) \hbox{~such that~} (1_S)\sigma=e\}
\end{equation*}
and define the family
\begin{equation*}
    \mathscr{H}^{\textit{top}}_1(S,T)=\{ H(e)\mid
    e\in\mathbf{E}^{\textit{top}}_1(S,T)\},
\end{equation*}
where by $H(e)$ we denote the maximal subgroup with the unity $e$
in the semigroup $T$. Also,
by $\mathfrak{TB}$ we denote the class
of all topological monoids $S$ with zero such that $S$ has
$\mathcal{B}^*$-property and every idempotent of $S$ lies in the
center of $S$.

We define the
category $\mathscr{T\!B}_{\operatorname{fin}}$ as
follows:
\begin{itemize}
    \item[$(i)$] $\operatorname{\textbf{Ob}}(\mathscr{T\!B}_{\operatorname{fin}})=
    \{(S,I)\mid S\in\mathfrak{TB} \textrm{~and~} I \textrm{~is a
    finite set}\}$, and if $S$ is a trivial
    semigroup then we identify $(S,I)$ and $(S,J)$ for all finite
    sets $I$ and $J$;
    \item[$(ii)$] $\operatorname{\textbf{Mor}}(\mathscr{T\!B}_{\operatorname{fin}})$
    consists of triples $(h,u,\varphi)\colon
    (S,I)\rightarrow(S^{\,\prime},I^{\,\prime})$, where
\begin{equation}\label{eq4-1}
\begin{split}
    & h\colon S\rightarrow S^{\,\prime} \textrm{~is a continuous
      homomorphism such that~}
      h\in \texttt{CHom}\,_0(S,S^{\,\prime}),\\
    & u\colon I\rightarrow H(e) \textrm{~is a map~}, \textrm{~for~} H(e)\in
      \mathscr{H}^{\textit{top}}_1(S,S^{\,\prime}),\\
    & \varphi\colon I\rightarrow I^{\,\prime} \textrm{~is an
    one-to-one map},
\end{split}
\end{equation}
with the composition defined by formulae (\ref{eq3-2}) and
(\ref{eq3-3}).
\end{itemize}
Straightforward verification shows that
$\mathscr{T\!B}_{\operatorname{fin}}$ is the category with the
identity morphism
$\varepsilon_{(S,I)}=(\operatorname{Id}_S,u_0,\operatorname{Id}_I)$
for any
$(S,I)\in\operatorname{\textbf{Ob}}(\mathscr{T\!B}_{\operatorname{fin}})$,
where $\operatorname{Id}_S\colon S\rightarrow S$ and
$\operatorname{Id}_I\colon I\rightarrow I$ are identity maps and
$(\alpha)u_0=1_S$ for all $\alpha\in I$.

We define a category
$\mathscr{B}^*_{\operatorname{fin}}(\mathscr{T\!S})$ as follows:
\begin{itemize}
    \item[$(i)$]
    Let
    $\operatorname{\textbf{Ob}}(\mathscr{B}^*_{\operatorname{fin}}(\mathscr{T\!S}))$
    be all finite topological Brandt $\lambda^0$-extensions of
    topological monoids $S$ with zeros
    such that $S$ has $\mathcal{B}^*$-property and every
    idempotent of $S$ lies in the center of $S$;
    \item[$(ii)$]
    Let
    $\operatorname{\textbf{Mor}}(\mathscr{B}^*_{\operatorname{fin}}(\mathscr{T\!S}))$
    be homomorphisms of finite topological Brandt $\lambda^0$-extensions
    of topological monoids $S$ with zeros such that $S$ has
    $\mathcal{B}^*$-property and every idempotent of $S$ lies in
    the center of $S$.
\end{itemize}

We define
a functor $\operatorname{\textbf{B}}$ from the category
$\mathscr{T\!B}_{\operatorname{fin}}$ into the category
$\mathscr{B}^*_{\operatorname{fin}}(\mathscr{T\!S})$ 
similarly as in
Section~\ref{sec3} (cf.  \ref{functor_B}).
Theorems~\ref{theorem3-1} and \ref{theorem4-6} imply:

\begin{theorem}\label{theorem4-15}
$\operatorname{\textbf{B}}$ is a full representative functor from
$\mathscr{T\!B}_{\operatorname{fin}}$ into
$\mathscr{B}^*_{\operatorname{fin}}(\mathscr{T\!S})$.
\end{theorem}

We define the first series of categories as follows:
\begin{itemize}
    \item[$(i)$] $\operatorname{\textbf{Ob}}(\mathscr{T\!B}_{\operatorname{fin}}\mathscr{I})=
    \{(S,I)\mid S\in\mathfrak{B} \textrm{~is a topological inverse monoid~and~} I \textrm{~is a
    non-empty finite} \\ \textrm{set}\}$, and if $S$ is a trivial semigroup then we
    identify $(S,I)$ and $(S,J)$ for all non-empty sets $I$ and
    $J$;
    \item[] $\operatorname{\textbf{Ob}}(\mathscr{CC\!T\!B}_{\operatorname{fin}}\mathscr{I})=
    \{(S,I)\mid S\in\mathfrak{B} \textrm{~is a countably compact topological inverse
    monoid~and~}\\
    I \textrm{~is a non-empty finite set}\}$, and if $S$ is a trivial semigroup then we
    identify $(S,I)$ and $(S,J)$ for all non-empty sets $I$ and
    $J$;
    \item[] $\operatorname{\textbf{Ob}}(\mathscr{C\!T\!B}_{\operatorname{fin}}\mathscr{I})=
    \{(S,I)\mid S\in\mathfrak{B} \textrm{~is a compact topological inverse
    monoid~and~} I \textrm{~is a non-} \\ \textrm{empty finite set}\}$, and if
    $S$ is a trivial semigroup then we
    identify $(S,I)$ and $(S,J)$ for all non-empty sets $I$ and
    $J$;

    \item[]
    $\operatorname{\textbf{Ob}}(\mathscr{\mathscr{T\!B}_{\operatorname{fin}}\mathscr{IC}})=
    \{(S,I)\mid S\textrm{~is a topological inverse Clifford monoid~and~} I \textrm{~is a
    non-empty} \\ \textrm{finite set}\}$, and if $S$ is a trivial semigroup then we
    identify $(S,I)$ and $(S,J)$ for all non-empty sets $I$ and
    $J$;
    \item[]
    $\operatorname{\textbf{Ob}}(\mathscr{\mathscr{CC\!T\!B}_{\operatorname{fin}}\mathscr{IC}})=
    \{(S,I)\mid S\textrm{~is a countably compact topological inverse Clifford monoid}
    \\ \textrm{and~} I \textrm{~is a
    non-empty finite set}\}$, and if $S$ is a trivial semigroup then we
    identify $(S,I)$ and $(S,J)$ for all non-empty sets $I$ and
    $J$;
    \item[]
    $\operatorname{\textbf{Ob}}(\mathscr{\mathscr{C\!T\!B}_{\operatorname{fin}}\mathscr{IC}})=
    \{(S,I)\mid S\textrm{~is a compact topological inverse Clifford monoid~and~} I \textrm{~is a }
     \\ \textrm{non-empty finite set}\}$, and if $S$ is a trivial semigroup then we
    identify $(S,I)$ and $(S,J)$ for all non-empty sets $I$ and
    $J$;

    \item[]
    $\operatorname{\textbf{Ob}}(\mathscr{\mathscr{T\!B}_{\operatorname{fin}}\mathscr{SL}})=
    \{(S,I)\mid S\textrm{~is a topological semilattice with unity and zero~and~} I$ ~is a
    non-empty finite set$\}$, and if $S$ is a trivial semigroup then we
    identify $(S,I)$ and $(S,J)$ for all non-empty sets $I$ and
    $J$;
    \item[]
    $\operatorname{\textbf{Ob}}(\mathscr{\mathscr{CC\!T\!B}_{\operatorname{fin}}\mathscr{SL}})=
    \{(S,I)\mid S\textrm{~is a countably compact topological semilattice with unity} \\
    \textrm{and zero~and~} I$ ~is a
    non-empty finite set$\}$, and if $S$ is a trivial semigroup then we
    identify $(S,I)$ and $(S,J)$ for all non-empty sets $I$ and
    $J$;
    \item[]
    $\operatorname{\textbf{Ob}}(\mathscr{\mathscr{C\!T\!B}_{\operatorname{fin}}\mathscr{SL}})=
    \{(S,I)\mid S\textrm{~is a compact topological semilattice with unity and zero and~}$
    $I$ ~is a non-empty finite set$\}$, and if $S$ is a trivial semigroup then we
    identify $(S,I)$ and $(S,J)$ for all non-empty sets $I$ and
    $J$;

    \item[]
    $\operatorname{\textbf{Ob}}(\mathscr{T\!B}_{\operatorname{fin}}\mathscr{S}_2)=
    \{(S,I)\mid S$~is a topological monoid with two idempotent
    zero and unity and $I$ is a non-empty finite set$\}$;

    \item[]
    $\operatorname{\textbf{Ob}}(\mathscr{T\!B}_{\operatorname{fin}}\mathscr{G})=
    \{(S,I)\mid S$~is a topological group and $I$ is a non-empty finite set$\}$;
    \item[]
    $\operatorname{\textbf{Ob}}(\mathscr{CC\!T\!B}_{\operatorname{fin}}\mathscr{G})=
    \{(S,I)\mid S$~is a countably compact topological group and $I$ is
    a non-empty finite set$\}$;
    \item[]
    $\operatorname{\textbf{Ob}}(\mathscr{C\!T\!B}_{\operatorname{fin}}\mathscr{G})=
    \{(S,I)\mid S$~is a compact topological group and $I$ is
    a non-empty finite set$\}$;

    \item[$(ii)$]
    $\operatorname{\textbf{Mor}}(\mathscr{T\!B}_{\operatorname{fin}}\mathscr{I})$,
    $\operatorname{\textbf{Mor}}(\mathscr{CC\!T\!B}_{\operatorname{fin}}\mathscr{I})$,
    $\operatorname{\textbf{Mor}}(\mathscr{C\!T\!B}_{\operatorname{fin}}\mathscr{I})$,
    $\operatorname{\textbf{Mor}}(\mathscr{\mathscr{T\!B}_{\operatorname{fin}}\mathscr{IC}})$,
    $\operatorname{\textbf{Mor}}(\mathscr{\mathscr{CC\!T\!B}_{\operatorname{fin}}\mathscr{IC}})$,
    $\operatorname{\textbf{Mor}}(\mathscr{\mathscr{C\!T\!B}_{\operatorname{fin}}\mathscr{IC}})$,
    $\operatorname{\textbf{Mor}}(\mathscr{\mathscr{T\!B}_{\operatorname{fin}}\mathscr{SL}})$,
    $\operatorname{\textbf{Mor}}(\mathscr{\mathscr{CC\!T\!B}_{\operatorname{fin}}\mathscr{SL}})$,
    $\operatorname{\textbf{Mor}}(\mathscr{\mathscr{C\!T\!B}_{\operatorname{fin}}\mathscr{SL}})$,
    \break $\operatorname{\textbf{Mor}}(\mathscr{T\!B}_{\operatorname{fin}}\mathscr{S}_2)$,
    $\operatorname{\textbf{Mor}}(\mathscr{T\!B}_{\operatorname{fin}}\mathscr{G})$,
    $\operatorname{\textbf{Mor}}(\mathscr{CC\!T\!B}_{\operatorname{fin}}\mathscr{G})$, and
    $\operatorname{\textbf{Mor}}(\mathscr{C\!T\!B}_{\operatorname{fin}}\mathscr{G})$
    consist of corresponding triples $(h,u,\varphi)\colon
    (S,I)\rightarrow(S^{\,\prime},I^{\,\prime})$, which satisfy
    condition (\ref{eq4-1}).
\end{itemize}
Obviously, these categories are subcategories of the
category $\mathscr{T\!B}_{\operatorname{fin}}$.
The second series of categories is defined as follows:
\begin{itemize}
    \item[$(i)$]
    Let
    $\operatorname{\textbf{Ob}}(\mathscr{B}^*_{\operatorname{fin}}(\mathscr{T\!I\!S}))$
    be all finite topological Brandt $\lambda^0$-extensions of
    topological inverse monoids $S$ with zeros in the class of
    topological inverse semigroups such that $S$ has
    $\mathcal{B}^*$-property and every
    idempotent of $S$ lies in the center of $S$; 
    \item[]
    Let
    $\operatorname{\textbf{Ob}}(\mathscr{B}^*(\mathscr{CC\!T\!I\!S}))$
    be all countably compact topological Brandt $\lambda^0$-extensions of
    topological inverse monoids $S$ with zeros in the class of
    topological inverse semigroups such that $S$ has
    $\mathcal{B}^*$-property and every
    idempotent of $S$ lies in the center of $S$; 
    \item[]
    Let
    $\operatorname{\textbf{Ob}}(\mathscr{B}^*(\mathscr{C\!T\!I\!S}))$
    be all compact topological Brandt $\lambda^0$-extensions of
    topological inverse monoids $S$ with zeros such that $S$ has
    $\mathcal{B}^*$-property and every
    idempotent of $S$ lies in the center of $S$; 

    \item[]
    Let $\operatorname{\textbf{Ob}}(\mathscr{B}^*_{\operatorname{fin}}(\mathscr{T\!ICS}))$
    be all finite topological Brandt $\lambda^0$-extensions of
    topological inverse Clifford monoids with zeros in the class
    of topological inverse semigroups; 
    \item[]
    Let
    $\operatorname{\textbf{Ob}}(\mathscr{B}^*(\mathscr{CCT\!ICS}))$
    be all countably compact topological Brandt $\lambda^0$-extensions of
    topological inverse Clifford monoids with zeros in the class
    of topological inverse semigroups; 
    \item[]
    Let
    $\operatorname{\textbf{Ob}}(\mathscr{B}^*(\mathscr{CT\!ICS}))$
    be all compact topological Brandt $\lambda^0$-extensions of
    topological inverse Clifford monoids with zeros; 

    \item[]
    Let
    $\operatorname{\textbf{Ob}}(\mathscr{B}^*(\mathscr{T\!SL}))$
    be all finite topological Brandt $\lambda^0$-extensions of topological
    semilattices with zeros and identities; 
    \item[]
    Let
    $\operatorname{\textbf{Ob}}(\mathscr{B}^*(\mathscr{CC\!T\!SL}))$
    be all countably compact topological Brandt $\lambda^0$-extensions of
    topological semilattices with zeros and identities in the class
    of topological inverse semigroups; 
    \item[]
    Let
    $\operatorname{\textbf{Ob}}(\mathscr{B}^*(\mathscr{C\!T\!SL}))$
    be all compact topological Brandt $\lambda^0$-extensions of
    topological semilattices with zeros and identities; 

    \item[]
    Let
    $\operatorname{\textbf{Ob}}(\mathscr{B}^*_{\operatorname{fin}}(\mathscr{T\!S}_2))$
    be all finite topological Brandt $\lambda^0$-extensions of topological
    monoids with two idempotents zeros and identities; 

    \item[]
    Let
    $\operatorname{\textbf{Ob}}(\mathscr{B}^*_{\operatorname{fin}}(\mathscr{T\!G}))$
    be all topological inverse Brandt semigroups with finite bands;   
    \item[]
    Let
    $\operatorname{\textbf{Ob}}(\mathscr{B}^*(\mathscr{CC\!T\!G}))$
    be all $0$-simple countably compact topological inverse semigroups;   
    \item[]
    Let
    $\operatorname{\textbf{Ob}}(\mathscr{B}^*(\mathscr{C\!T\!G}))$
    be all $0$-simple compact topological inverse semigroups;  

    \item[$(ii)$]
    Let
    $\operatorname{\textbf{Mor}}(\mathscr{B}^*_{\operatorname{fin}}(\mathscr{T\!I\!S}))$,
    $\operatorname{\textbf{Mor}}(\mathscr{B}^*(\mathscr{CC\!T\!I\!S}))$,
    $\operatorname{\textbf{Mor}}(\mathscr{B}^*(\mathscr{C\!T\!I\!S}))$,
    $\operatorname{\textbf{Mor}}(\mathscr{B}^*_{\operatorname{fin}}(\mathscr{T\!ICS}))$,
    \break
    $\operatorname{\textbf{Mor}}(\mathscr{B}^*(\mathscr{CCT\!ICS}))$,
    $\operatorname{\textbf{Mor}}(\mathscr{B}^*(\mathscr{CT\!ICS}))$,
    $\operatorname{\textbf{Mor}}(\mathscr{B}^*(\mathscr{T\!SL}))$,
    $\operatorname{\textbf{Mor}}(\mathscr{B}^*(\mathscr{CC\!T\!SL}))$,
    $\operatorname{\textbf{Mor}}(\mathscr{B}^*(\mathscr{C\!T\!SL}))$, and
    $\operatorname{\textbf{Mor}}(\mathscr{B}^*_{\operatorname{fin}}(\mathscr{T\!S}_2))$
    be continuous homomorphisms of corresponding topological Brandt
    $\lambda^0$-extensions of corresponding topological semigroups
    and let
    $\operatorname{\textbf{Mor}}(\mathscr{B}^*_{\operatorname{fin}}(\mathscr{T\!G}))$,
    $\operatorname{\textbf{Mor}}(\mathscr{B}^*(\mathscr{CC\!T\!G}))$,
    and $\operatorname{\textbf{Mor}}(\mathscr{B}^*(\mathscr{C\!T\!G}))$
    be non-trivial continuous
    homomorphisms of the
    corresponding topological inverse Brandt semigroups.
\end{itemize}

Theorem~\ref{theorem4-15} implies:

\begin{proposition}\label{proposition4-16}
$\textbf{B}$ is a full representative functor from
$\mathscr{T\!B}_{\operatorname{fin}}\mathscr{I}$ $[$ resp.,
$\mathscr{CC\!T\!B}_{\operatorname{fin}}\mathscr{I}$,
$\mathscr{C\!T\!B}_{\operatorname{fin}}\mathscr{I}$,
$\mathscr{\mathscr{T\!B}_{\operatorname{fin}}\mathscr{IC}}$,
$\mathscr{\mathscr{CC\!T\!B}_{\operatorname{fin}}\mathscr{IC}}$,
$\mathscr{\mathscr{C\!T\!B}_{\operatorname{fin}}\mathscr{IC}}$,
$\mathscr{\mathscr{T\!B}_{\operatorname{fin}}\mathscr{SL}}$,
$\mathscr{\mathscr{CC\!T\!B}_{\operatorname{fin}}\mathscr{SL}}$,
$\mathscr{\mathscr{C\!T\!B}_{\operatorname{fin}}\mathscr{SL}}$,
$\mathscr{T\!B}_{\operatorname{fin}}\mathscr{S}_2$, and
$\mathscr{T\!B}_{\operatorname{fin}}\mathscr{G}$ $]$ into
$\mathscr{B}^*_{\operatorname{fin}}(\mathscr{T\!I\!S})$ $[$ resp.,
$\mathscr{B}^*(\mathscr{CC\!T\!I\!S})$,
$\mathscr{B}^*(\mathscr{C\!T\!I\!S})$,
$\mathscr{B}^*_{\operatorname{fin}}(\mathscr{T\!ICS})$,
$\mathscr{B}^*(\mathscr{CCT\!ICS})$,
$\mathscr{B}^*(\mathscr{CT\!ICS})$,
$\mathscr{B}^*(\mathscr{T\!SL})$,
$\mathscr{B}^*(\mathscr{CC\!T\!SL})$,
$\mathscr{B}^*(\mathscr{C\!T\!SL})$,
$\mathscr{B}^*_{\operatorname{fin}}(\mathscr{T\!S}_2)$, and
$\mathscr{B}^*_{\operatorname{fin}}(\mathscr{T\!G})$ $]$.
\end{proposition}

Therefore Propositions~\ref{proposition3-3} and
~\ref{proposition4-16} imply:

\begin{corollary}\label{corollary4-17}
The categories
$\mathscr{\mathscr{T\!B}_{\operatorname{fin}}\mathscr{SL}}$, $[$
resp.,
$\mathscr{\mathscr{CC\!T\!B}_{\operatorname{fin}}\mathscr{SL}}$,
$\mathscr{\mathscr{C\!T\!B}_{\operatorname{fin}}\mathscr{SL}}$
$]$, and $\mathscr{B}^*(\mathscr{T\!SL})$ $[$ resp.,
$\mathscr{B}^*(\mathscr{CC\!T\!SL})$,
$\mathscr{B}^*(\mathscr{C\!T\!SL})$ $]$ are isomorphic.
\end{corollary}

Gutik and Repov\v{s} \cite{GutikRepovs2007} proved that every
countably compact $0$-simple topological inverse semigroup $S$
is
isomorphic to the
topological Brandt $\lambda^0$-extension
$\left(B^0_{\lambda}(G), \tau_{B}\right)$ of a topological group
$G$ in the class of topological inverse semigroups for some finite
cardinal $\lambda$. This implies
the following:

\begin{proposition}\label{proposition4-17}
$\textbf{B}$ is a full representative functor from
$\mathscr{CC\!T\!B}_{\operatorname{fin}}\mathscr{G}$ $[$resp.,
$\mathscr{C\!T\!B}_{\operatorname{fin}}\mathscr{G}$\ \!$]$ into
$\mathscr{B}^*(\mathscr{CC\!T\!G})$ $[$resp.,
$\mathscr{B}^*(\mathscr{C\!T\!G})$ $]$.
\end{proposition}

Comfort and Ross \cite{ComfortRoss1966} proved that the
Stone-\v{C}ech compactification of a pseudocompact topological
group is a topological group. Therefore the functor of the
Stone-\v{C}ech compactification \textbf{$\beta$} from the category
of pseudocompact topological groups back into itself determines a
monad.
Similarly, since the Stone-\v{C}ech compactification of a
countably compact $0$-simple topological inverse semigroup is a
compact $0$-simple topological inverse
semigroup~\cite[Theorem~3]{GutikRepovs2007}, we get the following:

\begin{corollary}\label{corollary4-18}
The functor of the Stone-\v{C}ech compactification $\beta\colon
\mathscr{B}^*(\mathscr{CC\!T\!G}) \rightarrow
\mathscr{B}^*(\mathscr{CC\!T\!G})$ determines a monad.
\end{corollary}


\section*{Acknowledgements}

This research was supported
by the Slovenian Research Agency grants
P1-0292-0101-04,  J1-9643-0101 and BI-UA/07-08/001. We acknowledge the Editor and the 
referees
for their comments and suggestions.
We also thank
Dikran Dikranjan for his comments.


\end{document}